\documentclass[11pt,reqno]{amsart}
\usepackage[T1]{fontenc}
\usepackage{lmodern}
\usepackage{microtype}
\usepackage{amssymb,mathtools,mathrsfs}
\usepackage{tikz-cd}
\usepackage[margin=1in]{geometry}
\usepackage{booktabs,array}
\usepackage[hidelinks]{hyperref}
\hypersetup{pdftitle={A Secondary Multiplicative Structure on Lawson Homology},pdfsubject={Lawson--Massey brackets, singular-homology-invisible torsion, and applications},pdfauthor={Wenchuan Hu}}
\numberwithin{equation}{section}
\newtheorem{theorem}{Theorem}[section]
\newtheorem{proposition}[theorem]{Proposition}
\newtheorem{lemma}[theorem]{Lemma}
\newtheorem{corollary}[theorem]{Corollary}
\theoremstyle{definition}
\newtheorem{definition}[theorem]{Definition}
\newtheorem{question}[theorem]{Question}
\theoremstyle{remark}
\newtheorem{remark}[theorem]{Remark}
\newcommand{\C}{\mathbb C}
\newcommand{\Z}{\mathbb Z}
\newcommand{\Q}{\mathbb Q}
\newcommand{\R}{\mathbb R}
\newcommand{\F}{\mathbb F}
\newcommand{\PP}{\mathbb P}
\newcommand{\CH}{\operatorname{CH}}
\newcommand{\NS}{\operatorname{NS}}
\newcommand{\Pic}{\operatorname{Pic}}

\newcommand{\cl}{\operatorname{cl}}
\newcommand{\im}{\operatorname{im}}

\newcommand{\mor}{\mathrm{mor}}
\newcommand{\sing}{\mathrm{sing}}
\newcommand{\sst}{\mathrm{sst}}
\newcommand{\alg}{\mathrm{alg}}
\newcommand{\homol}{\mathrm{hom}}

\newcommand{\an}{\mathrm{an}}

\newcommand{\MZ}{\mathsf M\Z}
\newcommand{\cA}{\mathcal A}
\newcommand{\cC}{\mathcal C}
\newcommand{\cZ}{\mathcal Z}
\newcommand{\one}{\mathbf 1}
\newcommand{\br}[1]{\langle #1\rangle}
\newcommand{\id}{\operatorname{id}}
\newcommand{\PD}{\operatorname{PD}}
\newcommand{\Map}{\operatorname{Map}}
\newcommand{\Spec}{\operatorname{Spec}}
\newcommand{\doi}[1]{\href{https://doi.org/#1}{doi:\nolinkurl{#1}}}
\newcommand{\arxiv}[1]{\href{https://arxiv.org/abs/#1}{arXiv:\nolinkurl{#1}}}
\allowdisplaybreaks[2]
\title[Secondary structure on Lawson homology]{A Secondary Multiplicative Structure on Lawson Homology}
\author{Wenchuan Hu}
\address{School of Mathematics, Sichuan University, Chengdu 610064, China}
\email{huwenchuan@gmail.com, wenchuan@scu.edu.cn}
\date{\today}
\begin{document}
\begin{abstract}
We introduce a new secondary operation on Lawson homology of smooth complex projective varieties. It refines the Lawson intersection product by retaining null-homotopy data that are invisible at the level of the graded Lawson intersection algebra. Intrinsically, the operation is the Toda bracket in multiplicative morphic cohomology transported through Friedlander--Lawson duality; in an associative cochain model it is represented by the classical Massey formula. For every defined triple, the affine bracket modulo its full indeterminacy gives a well-defined secondary invariant.

The new structure is genuinely nontrivial and strictly finer than its singular-homology realization. On a smooth projective eightfold we construct integral Lawson--Massey brackets whose quotient classes are nonzero, while specified values have exact order two and map to zero under the Lawson cycle map to singular homology. A square-zero perturbation makes all three associated singular-homology classes nonzero without changing the secondary values. Infinitely many such values remain independent even after quotienting by the sum of all their indeterminacies, over one fixed singular-homology triple. Projection-formula transfers preserve the phenomenon under projective bundles and odd-degree generically finite morphisms, and cyclic triple covers give examples with ample canonical bundle in every dimension at least eight.

The same defining systems yield nonzero motivic Massey values, scalar-indecomposable higher Chow torsion, and obstructions to integral and two-local formality. 
\end{abstract}
\maketitle

\section{Introduction}
Lawson homology of a complex projective variety is defined by
\[
 L_pH_k(X)=\pi_{k-2p}\cZ_p(X),\qquad k\geq 2p\geq0,
\]
where \(\cZ_p(X)\) is the topological abelian group of algebraic \(p\)-cycles. This homotopy group definition is primary throughout the paper. For a smooth projective variety, Friedlander--Lawson duality transports the multiplicative structure of morphic cohomology to the Lawson intersection product. Lawson homology also carries other classical operations, notably the unary \(s\)-operation coming from joins of cycles; see \cite{FriedlanderFiltrations}. The operation constructed here is different in kind: it is a partial, multivalued \emph{secondary} operation, defined only when adjacent primary products vanish and depending on coherent null-homotopies of those products.

\subsection*{The new secondary structure}
The principal structural contribution of this paper is a natural system of triple and higher Lawson--Massey brackets. Intrinsically these are Toda brackets in the multiplicative spectrum representing morphic cohomology, transported to Lawson homology by duality. In a cochain model they are the usual Massey brackets. Thus ``secondary multiplicative structure'' means the natural collection of these partially defined brackets, together with their full indeterminacies and naturality; we do \emph{not} claim a canonical \(A_\infty\)-algebra structure on the graded Lawson groups themselves.

For a defined triple \((x_1,x_2,x_3)\), the bracket is an affine coset in a Lawson group. After quotienting by the full indeterminacy, that coset becomes a single well-defined class; this is the associated \emph{Lawson--Massey secondary invariant}. It vanishes exactly when the bracket contains zero. In this sense the paper introduces both a new secondary operation and the quotient invariants extracted from it. To our knowledge, while primary Lawson operations and filtrations have been extensively studied, a Massey/Toda secondary operation on Lawson homology with its full indeterminacy and comparison to the Lawson cycle map has not previously been developed systematically.

The examples below show that this is not merely a formal repackaging of an existing product. The secondary quotient can be nonzero even when its image under the Lawson cycle map to the corresponding singular-homology Massey quotient is zero. Moreover, over one fixed nonzero singular-homology triple, the common quotient kernel contains an infinite direct sum of order-two secondary classes. The construction therefore detects a layer of algebraic cycle information not visible in ordinary singular homology.

Coefficient conventions are necessarily derived: finite-coefficient Lawson homology is generally not the tensor product of the integral Lawson group with the coefficient ring. The precise spectrum-level convention and universal coefficient sequence are given in Section~\ref{sec:models}.

\subsection*{Classical Massey products and their role}
Massey products were introduced by W.~S.~Massey as higher cohomology operations on singular cochains \cite{Massey1958}; systematic treatments of higher and matric products were developed by Kraines and May \cite{Kraines,May}. For cohomology classes \(\alpha,\beta,\gamma\) with \(\alpha\beta=\beta\gamma=0\), a choice of null-homotopies for the two adjacent cup products produces a triple bracket
\[
 \langle\alpha,\beta,\gamma\rangle
 \subset H^{|\alpha|+|\beta|+|\gamma|-1}(-;R),
\]
well defined up to the classical indeterminacy
\(\alpha H^{|\beta|+|\gamma|-1}+H^{|\alpha|+|\beta|-1}\gamma\). Thus a Massey product records secondary multiplicative information not determined by the graded cohomology ring alone. The standard geometric example is higher linking: triple Massey products detect the Borromean-type phenomenon in link complements even when all pairwise linking data vanish; see Massey's original higher-operation paper and his account of higher-order linking numbers \cite{Massey1958,MasseyLinking}.

Massey products also play a central role in rational homotopy theory. A defined bracket that does not contain zero obstructs formality of the relevant differential graded algebra, and Sullivan's theory makes such secondary operations part of the structure seen beyond the cohomology ring \cite{Sullivan}. In the opposite direction, the Deligne--Griffiths--Morgan--Sullivan formality theorem implies that all defined rational Massey products on compact K\"ahler manifolds contain zero \cite{DGMS}. In this paper the ``singular homology Massey bracket'' on a closed oriented manifold is the ordinary singular \emph{cohomology} bracket transported by Poincar\'e duality; it is not based on an independent multiplication on singular homology. This viewpoint is exactly the one used in the comparison theorem below.

\subsection*{Notation and conventions}
For a smooth complex variety \(V\), we write \(K_V=\omega_V=\det\Omega_V^1\) for its canonical line bundle, \(\Pic^0(V)\) for the identity component of the Picard scheme, and \(\NS(V)=\Pic(V)/\Pic^0(V)\) for the N\'eron--Severi group. We write \(\CH^r(V)_{\homol}\) and \(\CH^r(V)_{\alg}\) for the subgroups of codimension-\(r\) cycles that are homologically trivial and algebraically trivial, respectively. The symbol \(\boxtimes\) denotes exterior product. For a line bundle \(L\) on a smooth variety, \(c_1^{\mor}(L)\in L^1H^2(V;\Z)\) denotes its morphic first Chern class; under the standard identification \(L^1H^2(V;\Z)\cong\NS(V)\), it is the N\'eron--Severi class of \(L\) \cite[Theorem~8]{FLcocycles}. When no superscript is written, \(c_1(L)\) denotes the corresponding Chow or singular first Chern class according to context.

A secondary operation is not determined by a graded ring alone. It requires multiplication together with coherent null-homotopies. For smooth projective varieties, multiplicative morphic cohomology supplies this structure, and Friedlander--Lawson duality transfers it to Lawson homology. Heller's commutative motivic ring spectrum \(\MZ^{\sst}\) represents morphic cohomology \cite{Heller}. We make the integral linear structure explicit in Section~\ref{sec:models}; this also justifies the cochain calculations used in the new torsion example.

Throughout, a triple bracket is called nonzero when it \emph{does not contain zero}, equivalently when its class modulo its full indeterminacy is nonzero. This distinction is essential: exhibiting a nonzero element of a kernel of the Lawson cycle map is not enough to exhibit a nonzero Massey bracket.

\begin{theorem}[Construction and degree]\label{thm:intro-construction}
Let \(X\) be a smooth complex projective variety of dimension \(n\), let \(R=\Z,\Q,\R\), or \(\F_\ell\), with coefficients understood in the derived spectrum-level sense explained below, and let \(x_i\in L_{p_i}H_{k_i}(X;R)\). If
\(x_1\bullet x_2=x_2\bullet x_3=0\), there is a natural triple Lawson--Massey bracket
\[
 \br{x_1,x_2,x_3}_L
 \subset L_{p_1+p_2+p_3-2n}H_{k_1+k_2+k_3-4n+1}(X;R).
\]
It is a coset modulo
\begin{align*}
 I_L={}&x_1\bullet L_{p_2+p_3-n}H_{k_2+k_3-2n+1}(X;R)\\
 &+L_{p_1+p_2-n}H_{k_1+k_2-2n+1}(X;R)\bullet x_3.
\end{align*}
The assertion is an operation on ordinary Lawson groups when all displayed cycle indices are nonnegative. The morphic operation remains defined beyond that range.
\end{theorem}

\begin{theorem}[Comparison]\label{thm:intro-comparison}
The cycle map sends the bracket into the corresponding singular bracket, with singular homology brackets interpreted through Poincar\'e duality:
\[
 \Phi\bigl(\br{x_1,x_2,x_3}_L\bigr)
 \subseteq\br{\Phi(x_1),\Phi(x_2),\Phi(x_3)}_{\sing}.
\]
It also sends Lawson indeterminacy into singular indeterminacy, and hence induces a map between the two quotient classes.
\end{theorem}

\begin{theorem}[A bracket detected by singular homology]\label{thm:intro-ekedahl}
For some prime \(\ell\), there is a smooth complex projective threefold \(X\) and classes \(x_i\in L_2H_5(X;\F_\ell)\) with vanishing adjacent products and
\[
 0\notin\br{x_1,x_2,x_3}_L\subset L_0H_4(X;\F_\ell).
\]
The corresponding singular homology bracket is nonzero.
\end{theorem}

The stronger result is the following integral example.
\begin{theorem}[Integral torsion invisible to singular homology]\label{thm:main}
There is a smooth complex projective eightfold
\[
 X=S\times T\times B\times\PP^1
\]
and classes
\[
 x_1,x_2\in L_7H_{14}(X;\Z),\qquad
 x_3\in L_5H_{10}(X;\Z)
\]
such that the adjacent products vanish and
\begin{equation}\label{eq:main-bracket}
 0\notin\br{x_1,x_2,x_3}_L\subset L_3H_7(X;\Z).
\end{equation}
There is a value \(\xi\) of this bracket with
\begin{equation}\label{eq:main-value}
 \xi\neq0,\qquad 2\xi=0,\qquad \Phi_{3,7}(\xi)=0.
\end{equation}
Its Lawson quotient class is nonzero, whereas its image in the Poincar\'e-dual singular homology Massey quotient is zero. On the morphic side all three classes have positive cohomological degree, namely \(2,2,6\).
\end{theorem}

Here \(T\) is an Enriques surface chosen relative to a very general principally polarized abelian threefold \(B\). The other surface \(S\) is an equivariant complete-intersection quotient containing a prescribed elliptic curve. The order of the choices is important. We first choose \(B,T\), then an elliptic curve \(E\) relative to \(Y=T\times B\), and only then construct \(S\) containing \(E\). No simultaneous very-generality assertion about all factors is required.

The geometric cycle \(z=c_1(K_T)\boxtimes\gamma\) on \(Y\), where \(K_T=\omega_T\) is the canonical line bundle of the Enriques surface \(T\) and \(\gamma\) is homologically trivial and nonzero modulo two on \(B\), is supplied by established exterior-product results. The new calculation is that a torsion-linking defining system involving \(z\) survives \emph{all} global Massey indeterminacy. In particular, the proof does not rely on a Lawson K\"unneth decomposition.

The third singular homology class in this first construction is zero. The following strengthening removes that feature without increasing the dimension.

\begin{theorem}[Fixed nonzero singular homology classes]\label{thm:intro-fixed}
The eightfold \(X\) can be chosen with classes \(x_1,x_2\in L_7H_{14}(X;\Z)\) and \(x_{3,j}^+\in L_5H_{10}(X;\Z)\), for \(j\geq1\), such that every bracket \(\br{x_1,x_2,x_{3,j}^+}_L\) is defined and nonzero. Their three singular homology classes are nonzero and do not depend on \(j\). There are specified values \(\xi_j\in L_3H_7(X;\Z)\) of exact order two with \(\Phi(\xi_j)=0\).

If \(I_{L,j}\) is the full indeterminacy of the \(j\)-th bracket and \(I_\Sigma=\sum_{j\geq1}I_{L,j}\), the images of \(\xi_j\) are linearly independent modulo \(I_\Sigma+2L_3H_7(X;\Z)\). They generate a subgroup \(\bigoplus_{j\geq1}\Z/2\) in the kernel of the common quotient comparison to the ordinary Massey quotient of the fixed singular homology triple.
\end{theorem}

The key change is elementary but must be made at the level of defining systems. On the morphic side, replace the third class \(c\) by \(c^+=c+a h\vartheta\), where \(h=c_1(\mathcal O_{\PP^1}(1))\) and \(\vartheta=c_1^{\mor}(\Theta)\) for an ample line bundle \(\Theta\) defining a principal polarization on \(B\). The full indeterminacy is unchanged; the difference of the chosen values is a closed commutator multiplied by \(h^2\), and hence vanishes. Thus no strictly commutative integral cochain model is assumed. The resulting third singular cohomology class on the morphic side is the nonzero torsion class \(\cl(a)h\cl(\vartheta)\).

\begin{theorem}[Examples with ample canonical bundle]\label{thm:intro-general-type}
For every \(d\geq8\), a smooth connected complex projective \(d\)-fold with ample canonical bundle supports the fixed singular homology data phenomenon of Theorem~\ref{thm:intro-fixed}. The relevant Lawson groups are
\[
 L_{d-1}H_{2d-2},\quad L_{d-1}H_{2d-2},\quad L_{d-3}H_{2d-6},
\]
and its selected order-two values lie in \(L_{d-5}H_{2d-9}\), remain independent modulo the sum of all full indeterminacies and modulo two, and map to zero under the Lawson cycle map to singular homology.
\end{theorem}

The proof uses a transfer compatible with the two complete indeterminacy summands, followed by a cyclic cover of odd degree. Two further consequences are proved: the explicit degree-one higher Chow cycles underlying the defining systems give an infinite order-two subgroup modulo scalar-decomposable cycles, and both motivic and morphic cochain algebras of \(X\) fail to be formal over \(\Z\) and \(\Z_{(2)}\). These are applications of the specified secondary values, not claims that infinite higher Chow torsion or the general Massey obstruction to formality are new phenomena; see \cite{AZ,May}.

Sections~\ref{sec:models}--\ref{sec:comparison} give the construction and comparison. Sections~\ref{sec:ekedahl}--\ref{sec:criterion} retain the finite-coefficient example and the elliptic--quintic obstruction. Sections~\ref{sec:ingredients}--\ref{sec:eightfold} prove Theorem~\ref{thm:main}. Section~\ref{sec:nonzero-classes} proves Theorem~\ref{thm:intro-fixed}; Section~\ref{sec:transfer} proves Theorem~\ref{thm:intro-general-type}. Sections~\ref{sec:higher-chow}--\ref{sec:nonformality} give the motivic, higher Chow, and formality applications. The final section states the remaining limitations and questions.

\section{Multiplicative morphic cohomology, coefficients, and integral models}\label{sec:models}
\subsection{Lawson homotopy groups and derived coefficients}
To specify coefficients without replacing Lawson homotopy groups by ordinary singular homology, let \(\operatorname{Sing}_\bullet\cZ_p(X)\) denote the singular simplicial abelian group, with addition taken pointwise. No free abelian group on the set of singular simplices is introduced. Its normalized Moore complex computes the homotopy groups of \(\cZ_p(X)\) by the standard theory of simplicial abelian groups and the Dold--Kan correspondence; see \cite{MaySimplicial}. Via the standard equivalence between chain complexes of abelian groups and modules over the Eilenberg--Mac~Lane spectrum \(H\Z\) \cite{SchwedeShipley}, this gives a connective \(H\Z\)-module spectrum \(\mathbf L_p(X)\) satisfying
\[
 \pi_j\mathbf L_p(X)=\pi_j\cZ_p(X)\quad(j\ge0),\qquad \pi_j\mathbf L_p(X)=0\quad(j<0).
\]
Here ``connective'' has its standard stable-homotopy meaning, and for a ring \(R\), \(HR\) denotes the Eilenberg--Mac~Lane ring spectrum.

For a commutative ring \(R\), semicolon notation means spectrum-level derived extension,
\begin{equation}\label{eq:lawson-coeff-def}
 L_pH_k(X;R):=\pi_{k-2p}\bigl(\mathbf L_p(X)\mathbin{\wedge^{\mathbf L}_{H\Z}}HR\bigr),
\end{equation}
where \(\wedge^{\mathbf L}_{H\Z}\) is the derived relative smash product of module spectra \cite{HSS,SchwedeShipley}. This is a coefficient convention on the homotopy theory of the cycle space, not a redefinition of Lawson homology as singular-chain homology. If \(R\) is flat over \(\Z\), then
\[
 L_pH_k(X;R)\cong L_pH_k(X)\otimes_\Z R.
\]
For \(R=\F_\ell\) there is instead a natural exact sequence
\begin{equation}\label{eq:lawson-uct}
 0\to L_pH_k(X)/\ell\to L_pH_k(X;\F_\ell)
 \to L_pH_{k-1}(X)[\ell]\to0.
\end{equation}
It is the homotopy long exact sequence of \(H\Z\xrightarrow{\ell}H\Z\to H\F_\ell\). In particular, finite-coefficient Lawson homology must not in general be replaced by \(L_pH_k(X)\otimes\F_\ell\).

\subsection{The representing spectrum}
For a smooth complex quasi-projective variety \(V\), Heller's representation theorem gives
\begin{equation}\label{eq:represent}
 [\Sigma^\infty V_+,\Sigma^{m,q}\MZ^{\sst}]_{\mathrm{SH}(\C)}
 \cong L^qH^m(V;\Z).
\end{equation}
The multiplication is induced by exterior products of equidimensional cycle presheaves; see \cite[Proposition~5.6]{Heller}. It yields the usual product
\begin{equation}\label{eq:morphic-product}
 L^qH^m(V;R)\otimes L^{q'}H^{m'}(V;R)
 \longrightarrow L^{q+q'}H^{m+m'}(V;R).
\end{equation}
The diagonal of \(V\) turns external multiplication into internal multiplication.

To keep track of weights, use the ordinary mapping spectra
\[
 \cA_V(q)=\Map_{\mathrm{SH}(\C)}
   (\Sigma^\infty V_+,\Sigma^{0,q}\MZ^{\sst}),\qquad q\geq0.
\]
Then \(\pi_{-m}\cA_V(q)=L^qH^m(V;\Z)\). These spectra form a weight-graded multiplicative object. They retain the null-homotopies needed for secondary operations; their homotopy groups alone do not.

The weight-zero coefficient spectrum
\[
 K:=\Map_{\mathrm{SH}(\C)}(\one,\MZ^{\sst})
\]
has \(\pi_0K\cong\Z\) and \(\pi_jK=0\) for \(j\ne0\), because weight-zero morphic cohomology of a point is ordinary cohomology; see \cite{FLcocycles,Heller}. Hence \(K\simeq H\Z\) as an associative ring spectrum. The coefficient action therefore makes every \(\cA_V(q)\) an \(H\Z\)-module, compatibly with the weight-graded products.

For a commutative coefficient ring \(R\), we similarly set
\[
 L^qH^m(V;R):=\pi_{-m}\bigl(\cA_V(q)\mathbin{\wedge^{\mathbf L}_{H\Z}}HR\bigr).
\]
Here the coefficient object \(HR\) and the derived relative smash product are taken in the standard homotopy theory of module spectra \cite{HSS,SchwedeShipley}.
For flat \(R\) this is \(L^qH^m(V;\Z)\otimes R\). For \(R=\F_\ell\), the homotopy long exact sequence of multiplication by \(\ell\) gives
\begin{equation}\label{eq:coefficients}
 0\longrightarrow L^qH^m(V;\Z)/\ell
 \longrightarrow L^qH^m(V;\F_\ell)
 \xrightarrow{\beta} L^qH^{m+1}(V;\Z)[\ell]
 \longrightarrow0.
\end{equation}
This is the cohomological counterpart of \eqref{eq:lawson-uct}; the shift is reversed because \(L^qH^m=\pi_{-m}\).

\subsection{Why associative integral cochain calculations are legitimate}
\begin{lemma}[Integral linear model]\label{lem:linear-model}
For each smooth complex variety \(V\), the weight-graded mapping spectrum \(\bigoplus_q\cA_V(q)\) is naturally an associative algebra over \(H\Z\). Hence it admits an associative weight-graded differential graded \(\Z\)-algebra model \(\cC_V\) with
\[
 H^m(\cC_V(q))\cong L^qH^m(V;\Z).
\]
No strict commutativity of an integral differential graded model is asserted or used.
\end{lemma}
\begin{proof}
Heller constructs \(\MZ^{\sst}\) as a commutative strict motivic ring spectrum representing morphic cohomology \cite[Proposition~5.6]{Heller}. By the preceding identification \(K\simeq H\Z\), the coefficient action makes \(\bigoplus_q\cA_V(q)\), after passage to ordinary spectra, an associative \(H\Z\)-algebra. Shipley's Quillen equivalence between associative \(H\Z\)-algebra spectra and differential graded \(\Z\)-algebras \cite[Theorem~1.1]{Shipley} gives the asserted associative DGA model.
\end{proof}

Naturality under pullbacks, exterior products, and realization maps is understood intrinsically at the ring-spectrum/Toda-bracket level. Whenever a cochain formula is used for one fixed variety, we choose an associative DGA representative of that multiplicative theory; we do not require one simultaneous strict DGA model for an entire diagram of varieties. For factorwise exterior-product calculations we use the multiplicative maps of the underlying spectra (or derived tensor products after choosing models), and no K\"unneth quasi-isomorphism is assumed. When reduction modulo a prime is used at cochain level, a cofibrant degreewise-flat DGA representative is chosen first; this computes the spectrum-level derived extension above. In particular, the notation \(G/\ell\) always means the quotient \(G/\ell G\), not the full finite-coefficient group. The new Lawson bracket invisible to singular homology is integral; reduction modulo two is only a detector.

\subsection{Realization, duality, and pushforward}
Heller's monoidal factorization \(\id\to Q^{\sst}\to R\operatorname{Sing}_{\C}L\operatorname{Re}_{\C}\) \cite[Theorem~4.11(2)]{Heller} gives the multiplicative realization, inducing
\[
 \cl:L^qH^m(V;R)\longrightarrow H^m(V^{\an};R).
\]
If \(V\) is smooth projective of complex dimension \(n\), Friedlander--Lawson duality gives
\begin{equation}\label{eq:duality}
 D_V:L^qH^m(V;R)\xrightarrow{\ \sim\ }
 L_{n-q}H_{2n-m}(V;R)
\end{equation}
in the ordinary Lawson range \cite{FLduality}. The underlying graphing map is a weak homotopy equivalence of cocycle and cycle spaces, so the duality is compatible with the derived coefficient extension used above. We define
\[
 x\bullet y=D_V\bigl(D_V^{-1}x\cdot D_V^{-1}y\bigr).
\]
Thus
\[
 L_pH_k(V;R)\otimes L_{p'}H_{k'}(V;R)
 \longrightarrow L_{p+p'-n}H_{k+k'-2n}(V;R).
\]
The compatibility with topological Poincar\'e duality is
\begin{equation}\label{eq:compatibility}
 \Phi\circ D_V=\PD_V\circ\cl.
\end{equation}
We also use external products, proper pushforward, and the projection formula. In particular, for \(\pi:V\times\PP^1\to V\),
\begin{equation}\label{eq:projective-push}
 \pi_*:L^qH^m(V\times\PP^1)\longrightarrow L^{q-1}H^{m-2}(V),
 \qquad \pi_*h=1,
\end{equation}
where \(h=c_1(\mathcal O_{\PP^1}(1))\). These are compatible with duality and the projective bundle formula; see \cite{FG,HuLi}.

\section{Morphic Toda brackets and Massey representatives}\label{sec:brackets}
\subsection{Triple products and signs}
Let \(a_i\in L^{q_i}H^{m_i}(V;R)\) and assume
\begin{equation}\label{eq:adjacent}
 a_1a_2=0,\qquad a_2a_3=0.
\end{equation}
Choose representatives and null-homotopies in the multiplicative model of Section~\ref{sec:models}. Their two composites with the outer classes have opposite boundaries; gluing gives a class of weight \(q_1+q_2+q_3\) and cohomological degree \(m_1+m_2+m_3-1\). The set of resulting values is
\begin{equation}\label{eq:morphic-bracket}
 \br{a_1,a_2,a_3}_{\mor}
 \subset L^{q_1+q_2+q_3}H^{m_1+m_2+m_3-1}(V;R).
\end{equation}
This is the ring-spectrum Toda operation, expressed in cohomological grading; see \cite{BauesMuro,May,Toda}.

In an associative cochain model choose closed \(A_i\) representing \(a_i\), and choose \(U,V\) with
\[
 dU=A_1A_2,\qquad dV=A_2A_3.
\]
Our sign convention is
\begin{equation}\label{eq:cochain-bracket}
 W=A_1V+(-1)^{m_1+1}UA_3.
\end{equation}

\begin{lemma}[Full triple indeterminacy]\label{lem:indeterminacy}
The cochain \(W\) is closed, and \eqref{eq:morphic-bracket} is an affine coset modulo
\begin{align}
 I_{\mor}={}&a_1L^{q_2+q_3}H^{m_2+m_3-1}(V;R)\nonumber\\
 &+L^{q_1+q_2}H^{m_1+m_2-1}(V;R)a_3.
 \label{eq:morphic-indeterminacy}
\end{align}
In particular, the bracket determines a well-defined quotient class.
\end{lemma}
\begin{proof}
The derivation rule gives
\[
 dW=(-1)^{m_1}A_1A_2A_3+(-1)^{m_1+1}A_1A_2A_3=0.
\]
Changing \(U,V\) by closed cochains \(U_0,V_0\) changes the value by
\(a_1[V_0]+(-1)^{m_1+1}[U_0]a_3\). Every element of the displayed subgroup arises in this way. Changing the representatives produces boundaries and the same two kinds of terms. The equivalent ring-spectrum description says that the difference between two null-homotopies is a loop in the appropriate mapping space; its group of components is exactly the corresponding cohomology group in \eqref{eq:morphic-indeterminacy}.
\end{proof}

\begin{definition}
A defined triple bracket \emph{vanishes} when it contains zero. This is equivalent to vanishing of its quotient class modulo \eqref{eq:morphic-indeterminacy}.
\end{definition}

\begin{proposition}[Naturality]\label{prop:naturality}
For a multiplicative map of the models, including pullback by a morphism of smooth varieties,
\[
 f^*\br{a_1,a_2,a_3}_{\mor}
 \subseteq\br{f^*a_1,f^*a_2,f^*a_3}_{\mor}.
\]
\end{proposition}
\begin{proof}
Apply the map to the representatives and their chosen null-homotopies. It sends the glued representative to the representative for the transported defining system. Varying the defining system gives the inclusion.
\end{proof}

\subsection{Higher operations}
The same multiplicative model supports higher Massey--Toda operations. In a differential graded presentation one uses a compatible defining system in the sense of May \cite{May}. For an \(r\)-fold operation, an entry corresponding to the classes \(a_i,\ldots,a_j\) has weight
\(\sum_{t=i}^j q_t\) and degree \(\sum_{t=i}^j m_t-(j-i)\). Hence, whenever a full defining system exists,
\begin{equation}\label{eq:higher-degree}
 \br{a_1,\ldots,a_r}_{\mor}
 \subset L^{\sum_iq_i}H^{\sum_im_i-(r-2)}(V;R).
\end{equation}
Pairwise vanishing suffices only for \(r=3\); for \(r\geq4\), compatible choices of lower null-homotopies are required. We do not assert that the indeterminacy of a higher bracket is an affine subgroup of the simple form \eqref{eq:morphic-indeterminacy}. All nonvanishing arguments in this paper concern triple products.

\section{Transfer to Lawson homology}\label{sec:lawson}
Let \(X\) be smooth projective of dimension \(n\). For \(x_i\in L_{p_i}H_{k_i}(X;R)\), put
\[
 a_i=D_X^{-1}x_i\in L^{n-p_i}H^{2n-k_i}(X;R).
\]
\begin{definition}
When the adjacent intersection products vanish, define
\[
 \br{x_1,x_2,x_3}_L=D_X\br{a_1,a_2,a_3}_{\mor}.
\]
\end{definition}

\begin{theorem}[Degree and indeterminacy]\label{thm:lawson-degree}
The bracket takes values in
\[
 L_{p_1+p_2+p_3-2n}H_{k_1+k_2+k_3-4n+1}(X;R)
\]
and its indeterminacy is
\begin{align*}
 I_L={}&x_1\bullet L_{p_2+p_3-n}H_{k_2+k_3-2n+1}(X;R)\\
 &+L_{p_1+p_2-n}H_{k_1+k_2-2n+1}(X;R)\bullet x_3.
\end{align*}
\end{theorem}
\begin{proof}
The total morphic weight and degree are
\[
 q=3n-(p_1+p_2+p_3),\qquad
 m=6n-(k_1+k_2+k_3)-1.
\]
Duality sends these to cycle index \(n-q=p_1+p_2+p_3-2n\) and homological degree
\(2n-m=k_1+k_2+k_3-4n+1\). Applying the same substitution to the two groups in \eqref{eq:morphic-indeterminacy} gives the asserted formula for \(I_L\).
\end{proof}

\begin{corollary}[Higher Lawson degree]
An \(r\)-fold Lawson bracket, whenever defined, has target
\[
 L_{\sum_i p_i-(r-1)n}
 H_{\sum_i k_i-2(r-1)n+(r-2)}(X;R).
\]
\end{corollary}
\begin{proof}
Apply duality to \eqref{eq:higher-degree}.
\end{proof}

\begin{remark}[Smoothness and range]
The smoothness assumption supplies the duality isomorphism, not merely a convenient presentation of the product. For a singular projective variety, the argument does not define an intersection-based bracket on all ordinary Lawson groups. One must instead work in an appropriate cohomological or bivariant theory, or supply an additional duality hypothesis. Similarly, a negative cycle index is not silently interpreted as an ordinary Lawson group.
\end{remark}

\section{Comparison with singular homology}\label{sec:comparison}
For a closed oriented manifold \(M\), define its singular homology bracket by
\[
 \br{u_1,u_2,u_3}_{\sing}
 =\PD_M\br{\PD_M^{-1}u_1,\PD_M^{-1}u_2,\PD_M^{-1}u_3}_{C^*(M;R)}.
\]
\begin{theorem}[Cycle-map comparison]\label{thm:comparison}
For \(X\) and \(x_i\) as in Theorem~\ref{thm:lawson-degree},
\[
 \Phi\bigl(\br{x_1,x_2,x_3}_L\bigr)
 \subseteq\br{\Phi(x_1),\Phi(x_2),\Phi(x_3)}_{\sing}.
\]
Moreover, \(\Phi(I_L)\subseteq I_{\sing}\).
\end{theorem}
\begin{proof}
Apply Proposition~\ref{prop:naturality} to the multiplicative realization of Section~\ref{sec:models}. In the Eilenberg--Mac Lane target (equivalently, in the associated derived category of chain complexes \cite{SchwedeShipley}), its cochain expression is precisely the ordinary Massey formula \eqref{eq:cochain-bracket}. Equation~\eqref{eq:compatibility} transfers the resulting inclusion to homology. Multiplicativity sends each of the two indeterminacy terms into the corresponding singular term.
\end{proof}

\begin{remark}
Naturality gives an inclusion, not in general equality. Singular null-homotopies need not lift to morphic null-homotopies. This distinction is one possible source of a finer secondary invariant.
\end{remark}

\begin{corollary}[Rational comparison]\label{cor:rational-formality}
On a smooth complex projective variety, every defined rational singular Massey product vanishes. Thus a nonzero rational Lawson quotient class, if constructed, would automatically map to zero in the corresponding rational singular-homology Massey quotient.
\end{corollary}
\begin{proof}
The analytic manifold is compact K\"ahler. The formality theorem of Deligne--Griffiths--Morgan--Sullivan \cite{DGMS} implies the vanishing of its defined rational Massey products. Now use Theorem~\ref{thm:comparison} and pass to quotients by indeterminacy.
\end{proof}

\section{A nonzero bracket detected by singular homology}\label{sec:ekedahl}
We retain the finite-coefficient construction because it provides a different kind of nonvanishing from the integral example below.

\begin{theorem}[Ekedahl]\label{thm:ekedahl}
For some prime \(\ell\), there is a smooth complex projective surface \(S_0\) and classes \(\alpha_i\in H^1(S_0;\F_\ell)\) such that
\[
 \alpha_1\alpha_2=\alpha_2\alpha_3=0,
 \qquad 0\notin\br{\alpha_1,\alpha_2,\alpha_3}\subset H^2(S_0;\F_\ell).
\]
\end{theorem}
This is Ekedahl's construction \cite{Ekedahl}; see also \cite[Theorem~120(i)]{Catanese}. It uses finite-group cohomology and a projective-surface realization of a finite fundamental group.

\begin{lemma}[Low-weight comparison]\label{lem:low-weight}
For every smooth complex projective variety \(V\),
\[
 L^1H^1(V;\Z)=H^1(V;\Z),\qquad L^1H^2(V;\Z)=\NS(V),
\]
and \(L^1H^1(V;\F_\ell)\to H^1(V;\F_\ell)\) is an isomorphism. For a smooth projective surface \(S_0\), the map
\[
 L^2H^2(S_0;\F_\ell)\longrightarrow H^2(S_0;\F_\ell)
\]
is also an isomorphism.
\end{lemma}
\begin{proof}
The integral weight-one computation is \cite[Theorem~8]{FLcocycles}. The exponential sequence identifies the torsion of \(\NS(V)\) with the torsion of \(H^2(V;\Z)\): the map to \(H^2(V,\mathcal O_V)\) kills torsion, and \(\Pic^0(V)\) is the kernel of the integral first Chern class. Comparing \eqref{eq:coefficients} with the corresponding singular coefficient sequence gives the finite-coefficient weight-one isomorphism.

For a surface, duality and Dold--Thom give
\[
 L^2H^2(S_0;\F_\ell)\cong L_0H_2(S_0;\F_\ell)
 \cong H_2(S_0;\F_\ell).
\]
Equation~\eqref{eq:compatibility} identifies the cycle map with Poincar\'e duality, proving the last assertion.
\end{proof}

\begin{theorem}\label{thm:ekedahl-lawson}
The threefold \(X_\ell=S_0\times\PP^1\) has classes
\(x_i\in L_2H_5(X_\ell;\F_\ell)\) with
\[
 0\notin\br{x_1,x_2,x_3}_L\subset L_0H_4(X_\ell;\F_\ell).
\]
\end{theorem}
\begin{proof}
Lift the \(\alpha_i\) uniquely to \(\widetilde\alpha_i\in L^1H^1(S_0;\F_\ell)\). Their adjacent products lie in \(L^2H^2(S_0;\F_\ell)\) and vanish by Lemma~\ref{lem:low-weight}, since their cycle classes vanish. The morphic bracket is therefore defined, and comparison with Ekedahl's bracket shows it does not contain zero.

Pull back along \(\pi:X_\ell\to S_0\). This projection has a section. If the pulled-back morphic bracket contained zero, restriction along the section and naturality would put zero in the original bracket on \(S_0\), a contradiction. Duality in dimension three sends these classes \(L^1H^1\) to \(L_2H_5\) and the target \(L^3H^2\) to \(L_0H_4\).
\end{proof}

\begin{corollary}
For every \(r\geq1\), \(S_0\times\PP^r\) admits a nonzero bracket
\[
 \bigl(L_{r+1}H_{2r+3}(-;\F_\ell)\bigr)^3
 \dashrightarrow L_{r-1}H_{2r+2}(-;\F_\ell).
\]
\end{corollary}
\begin{proof}
Repeat the pullback-and-section argument and apply duality with \(n=r+2\).
\end{proof}

For \(r=1\), the target is \(L_0H_4\cong H_4\). Thus the construction proves nonvanishing but not invisibility under the Lawson cycle map. Raising the Lawson cycle index by adding projective factors does not change that distinction.

\section{What invisibility under the Lawson cycle map requires}\label{sec:criterion}
Put \(Q=q_1+q_2+q_3\) and \(M=m_1+m_2+m_3-1\). A defined triple bracket determines a class in
\begin{equation}\label{eq:quotient-map}
 L^QH^M(X;R)/I_{\mor}
 \longrightarrow H^M(X^{\an};R)/I_{\sing}.
\end{equation}
\begin{definition}\label{def:invisible}
A morphic triple bracket is \emph{invisible to singular cohomology} if its class in the morphic quotient on the left of \eqref{eq:quotient-map} is nonzero and its image in the singular-cohomology quotient on the right is zero. For a smooth projective variety, the Lawson-dual bracket is \emph{invisible to singular homology} if its Lawson quotient class is nonzero while its image under the Lawson cycle map is zero in the Poincar\'e-dual singular-homology Massey quotient. By \eqref{eq:compatibility}, these formulations are equivalent under Friedlander--Lawson and Poincar\'e duality.
\end{definition}
A nonzero element in the kernel of the Lawson cycle map need not meet this definition: it may lie in the full indeterminacy, or it may fail to arise from a Massey defining system. The example of Sections~\ref{sec:ingredients}--\ref{sec:eightfold} addresses both issues.

\begin{proposition}[The elliptic--quintic divisor obstruction]\label{prop:quintic}
Let \(E_0\) be an elliptic curve, let \(V\subset\PP^4\) be a smooth quintic threefold, and put \(Z=E_0\times V\). Every defined triple morphic Massey product with all three classes in \(L^1H^2(Z;\Q)\) contains zero. Consequently these classes cannot produce a Lawson bracket invisible to singular homology in \(L_1H_3(Z;\Q)\), even in instances where that group's kernel of the Lawson cycle map is infinite-dimensional.
\end{proposition}
\begin{proof}
Since \(\Pic(V)=\Z\) and \(H^1(V,\mathcal O_V)=0\),
\[
 L^1H^2(Z;\Q)=\NS(Z)_\Q=\Q e\oplus\Q h,
\]
where \(e\) comes from a point divisor on \(E_0\) and \(h\) from the hyperplane class on \(V\). In singular cohomology, \(e^2=0\), while \(eh,h^2\) are nonzero and linearly independent. If \(u=ae+bh\) and \(v=ce+dh\) are nonzero and \(uv=0\) morphically, then their singular cohomology product gives
\[
 ad+bc=0,\qquad bd=0.
\]
These equations force \(b=d=0\). Thus every pair of nonzero adjacent classes is made of multiples of \(e\).

If one class in a triple is zero, represent it by zero and take the null-homotopy of the product adjacent to that zero class to be zero where needed; the other required null-homotopy may be arbitrary. Formula~\eqref{eq:cochain-bracket} then supplies a value zero. Otherwise all three classes are multiples of \(e\) and are pulled back from \(E_0\). Restriction along a section verifies adjacent-product vanishing on \(E_0\). Their bracket there has target \(L^3H^5(E_0;\Q)=H^5(E_0;\Q)=0\), by the stable-weight comparison \cite[Theorems~5.6 and 5.8]{FLduality}. Pullback gives zero in the bracket on \(Z\). Finally, duality for the fourfold \(Z\) identifies its target \(L^3H^5(Z;\Q)\) with \(L_1H_3(Z;\Q)\). Examples with the large kernel mentioned in the statement occur in \cite{Hu}.
\end{proof}

\section{Geometric ingredients for the integral example}\label{sec:ingredients}
All groups in Sections~\ref{sec:ingredients}--\ref{sec:scope} have integral coefficients unless another coefficient group is explicitly displayed. Write
\[
 A^r(V):=\CH^r(V)/\CH^r(V)_{\alg}=L^rH^{2r}(V;\Z),
\]
where the last identification is the degree-\(2r\) morphic-cycle description modulo algebraic equivalence; see \cite{FLcocycles}. Thus \(A^r(V)\) is the group of codimension-\(r\) cycles modulo algebraic equivalence. A bar, as in \(\bar v\), denotes reduction modulo two in the indicated integral group.

\subsection{A homologically trivial class nonzero modulo two}
\begin{lemma}\label{lem:gamma}
For a very general principally polarized abelian threefold \(B\), there exists
\(\gamma\in\CH^2(B)_{\homol}\) whose image in \(A^2(B)/2\) is nonzero. In fact, the images of homologically trivial cycles span an infinite-dimensional subspace of \(A^2(B)/2\).
\end{lemma}
\begin{proof}
Totaro proves that \(\CH^2(B)/2\) is infinite \cite[Theorem~3.1]{Totaro}. Let
\[
 G=\CH^2(B),\quad K=\ker(G\xrightarrow{\cl}H^4(B;\Z)),\quad
 J=\im(G\xrightarrow{\cl}H^4(B;\Z)).
\]
The group \(J\) is finitely generated, being a subgroup of the finitely generated group \(H^4(B;\Z)\). Right exactness of tensor product gives
\[
 K/2\longrightarrow G/2\longrightarrow J/2\longrightarrow0.
\]
The middle group is an infinite-dimensional \(\F_2\)-vector space and the last is finite-dimensional. Hence the image of \(K/2\) is infinite-dimensional.

Algebraically trivial cycles over \(\C\) form a divisible group. Indeed, differences in an algebraic family can be taken over a smooth projective curve, and multiplication by two is surjective on its Jacobian. Correspondences carry those divisions to divisions of the cycle. Thus \(\CH^2(B)_{\alg}\subset2\CH^2(B)\), and the quotient map induces an isomorphism
\(\CH^2(B)/2\cong A^2(B)/2\). This proves the assertion.
\end{proof}

\subsection{A nondivisible torsion cycle on a fivefold}
We use the following published result, in precisely its modulo-two form.
\begin{theorem}[Schreieder's exterior product theorem]\label{thm:schreieder}
Let \(V\) be a smooth complex projective variety and let \(T\) be an Enriques surface very general with respect to \(V\). Write \(K_T=\omega_T\) for its canonical line bundle; for an Enriques surface, \(K_T\) is the nontrivial line bundle of order two. Then
\[
 A^r(V)/2\longrightarrow A^{r+1}(T\times V)/2,
 \qquad \bar v\longmapsto\overline{c_1(K_T)\boxtimes v}
\]
is injective. The indicated integral product is killed by two.
\end{theorem}
This is \cite[Theorem~1.3]{Schreieder}, including its nondivisibility assertion. Alexandrou's \cite[Theorem~6.1]{Alexandrou} gives a compatible Chow-modulo-\(n\) construction for arbitrary integers \(n\); that generalization is not needed for the present order-two example.

Choose \(B\) as in Lemma~\ref{lem:gamma}, choose \(T\) as in Theorem~\ref{thm:schreieder} relative to \(B\), and set
\begin{equation}\label{eq:Y-z}
 Y=T\times B,\qquad t_T:=c_1^{\mor}(K_T)\in L^1H^2(T),\qquad
 z=t_T\boxtimes\gamma\in L^3H^6(Y).
\end{equation}
Here \(\gamma\) also denotes its image in morphic cohomology. Then
\begin{equation}\label{eq:z-properties}
 2z=0,\qquad \cl(z)=0,\qquad
 \bar z\neq0\quad\text{in }L^3H^6(Y)/2.
\end{equation}
The first two statements follow from \(K_T^{\otimes2}\cong\mathcal O_T\) and \(\cl(\gamma)=0\). The third follows from Theorem~\ref{thm:schreieder}. Nonzero torsion alone would not be sufficient: the nonzero reduction modulo two will be used in the detection theorem.

\subsection{An elliptic exterior-product detector}
The second exterior-product ingredient is the following special case of a theorem of Alexandrou--Zhou.
\begin{theorem}[Elliptic detector]\label{thm:AZ}
For every smooth complex projective variety \(Y\), there is an elliptic curve \(E\) such that, for every nonzero
\(\bar u\in L^1H^1(E)/2\) and all integers \(c,p\geq0\), the map
\begin{equation}\label{eq:AZ}
 \F_2\bar u\otimes_{\F_2}\bigl(L^cH^{2c-p}(Y)/2\bigr)
 \longrightarrow L^{c+1}H^{2c+1-p}(E\times Y)/2
\end{equation}
induced by exterior product is injective.
\end{theorem}
We invoke \cite[Corollary~6.3]{AZ}, specialized to the prime two.ootnote{Only the \(n=2\) specialization is used here. In the proof chain of \cite{AZ}, the Poincar\'e-duality argument in Lemma~6.1 requires a primitive (equivalently, free rank-one) class for composite \(n\); Corollary~6.3 is formulated with a free rank-one \(\Z/n\)-submodule, and for \(n=2\) every nonzero class is automatically primitive. Thus this issue does not affect the specialization used in the present paper.} In its construction, one can take a countable algebraically closed field of definition for \(Y\) and an elliptic curve with transcendental \(j\)-invariant over that field. The theorem applies to \emph{every} nonzero line in \(L^1H^1(E)/2\), so the line that will arise from the double-cover construction below need not be chosen in advance.

Take \(E\) for the already chosen fivefold \(Y\). Equations~\eqref{eq:z-properties} and \eqref{eq:AZ}, with \(c=3,p=0\), show that
\begin{equation}\label{eq:uz-nonzero}
 \bar u\neq0\quad\Longrightarrow\quad
 \overline{u\boxtimes z}\neq0
 \text{ in }L^4H^7(E\times Y)/2.
\end{equation}
This is the only elliptic exterior-product assertion needed in the proof.

\subsection{A surface extending the prescribed double cover}
The surface providing indeterminacy control is not the Enriques factor \(T\). It is constructed after \(E\) is fixed.

\begin{lemma}[Relative equivariant complete intersection]\label{lem:surface}
Let \(E\) be an elliptic curve, and let \(\widetilde E\to E\) be a connected \emph{\'etale} double cover. There exist a smooth connected projective surface \(S\), a closed embedding \(f:E\hookrightarrow S\), and a connected \'etale double cover \(\widetilde S\to S\) whose pullback to \(E\) is the prescribed cover, such that
\begin{equation}\label{eq:H1S}
 H^1(S;\Z)=0.
\end{equation}
Let \(q:\widetilde S\to S\) denote the double cover. Its associated order-two line bundle \(\lambda_S\in\Pic(S)[2]\) is characterized by the eigensheaf decomposition
\[
 q_*\mathcal O_{\widetilde S}\cong\mathcal O_S\oplus\lambda_S^{-1},\qquad \lambda_S^{\otimes2}\cong\mathcal O_S.
\]
Set \(t_S:=c_1^{\mor}(\lambda_S)\in L^1H^2(S)\). Then
\begin{equation}\label{eq:tS-properties}
 t_S\neq0,\qquad 2t_S=0,\qquad f^*t_S=0\text{ in }L^1H^2(E).
\end{equation}
Let
\[
 \beta:L^1H^1(S;\F_2)\longrightarrow L^1H^2(S;\Z)[2]
\]
be the Bockstein connecting homomorphism in \eqref{eq:coefficients}, induced by the coefficient cofiber sequence \(H\Z\xrightarrow{2}H\Z\to H\F_2\) in the stable category of \(H\Z\)-modules \cite{Adams,SchwedeShipley}. By the \emph{Bockstein preimage} of \(t_S\) we mean a class \(\alpha_S\) with \(\beta(\alpha_S)=t_S\). In the present situation it is unique, and it satisfies
\begin{equation}\label{eq:cover-restriction}
 f^*\alpha_S\neq0\quad\text{in }L^1H^1(E;\F_2)=H^1(E;\F_2).
\end{equation}
\end{lemma}
\begin{proof}
Let \(\iota\) denote the deck involution. Choose a line bundle \(M\) of degree four on \(E\). Its pullback embeds \(\widetilde E\) in \(\PP^7\), since it has degree eight on an elliptic curve. If \(\eta\in\Pic(E)[2]\setminus\{0\}\) defines the cover, then
\[
 H^0(\widetilde E,\pi_E^*M)
 =H^0(E,M)\oplus H^0(E,M\otimes\eta).
\]
Both summands have dimension four and are the two eigenspaces of \(\iota\). The fixed locus of the resulting involution on \(\PP^7\) is therefore
\(F=\PP^3\sqcup\PP^3\). The embedded \(\widetilde E\) is disjoint from \(F\), since the embedding is equivariant and \(\iota\) has no fixed point on \(\widetilde E\).

For sufficiently large even \(d\), consider the invariant vector space
\[
 V_d=H^0(\PP^7,\mathcal I_{\widetilde E}(d))^{\iota}.
\]
Its sections are base-point-free away from \(\widetilde E\), and along the curve they generate the descended conormal bundle after twisting. Here is a justification that includes the fixed locus. The even power \(\mathcal O_{\PP^7}(d)\) descends to an ample line bundle on the finite quotient \(\PP^7/\langle\iota\rangle\); for large \(d\), Serre vanishing gives generation of the ideal of the image of \(\widetilde E\), and of its conormal bundle in the smooth free locus. Pulling back gives the claims. Equivalently, one can separate jets along the two-point free orbits and average; at fixed points the action on the even-degree fibre is trivial.

Take five general sections in \(V_d\). Off \(\widetilde E\cup F\), their common zero locus is smooth of codimension five by Bertini on the free quotient. Along \(\widetilde E\), its conormal bundle in \(\PP^7\) has rank six. At a point of the curve, the locus where five conormal values have rank at most four has codimension
\[
 (6-4)(5-4)=2
\]
in the space of five-tuples. Since the curve has dimension one, the incidence variety of such rank failures has positive codimension in the parameter space. Thus a general five-tuple has rank five at every point of \(\widetilde E\). Finally, at each point of either fixed \(\PP^3\), simultaneous vanishing of five sections has codimension five. The corresponding incidence over a threefold cannot dominate the parameter space. The five-tuple may therefore also be chosen with no common zero on \(F\).

Their complete intersection \(\widetilde S\subset\PP^7\) is a smooth surface, contains \(\widetilde E\), and has a free involution. It is connected, and its Koszul resolution gives
\(H^1(\widetilde S,\mathcal O_{\widetilde S})=0\). Set
\(S=\widetilde S/\langle\iota\rangle\). The quotient is smooth and projective and contains \(E\) as a closed subvariety. Its pulled-back cover is \(\widetilde E\to E\). The trace splitting gives \(H^1(S,\mathcal O_S)=0\), and Hodge theory gives \(H^1(S;\Z)=0\).

The line bundle \(\lambda_S\) defined by the anti-invariant eigensheaf above is nontrivial of order two: if \(\lambda_S\cong\mathcal O_S\), the corresponding degree-two \'etale algebra would split and the cover would be disconnected. Since \(\Pic^0(S)=0\), its class \(t_S\) is nonzero in \(\NS(S)=L^1H^2(S)\). On the curve \(E\), its restriction has degree zero, and hence its morphic divisor class is zero. This proves \eqref{eq:tS-properties}.

By \eqref{eq:coefficients} and \eqref{eq:H1S}, the connecting map \(\beta\) is an isomorphism onto the two-torsion subgroup containing \(t_S\), so \(t_S\) has a unique Bockstein preimage \(\alpha_S\). Under Lemma~\ref{lem:low-weight}, \(\alpha_S\) is the class in \(H^1(S;\F_2)\) classifying the connected \'etale double cover; compatibility of the coefficient Bockstein with the exponential sequence sends this class to \(c_1(\lambda_S)\). Its restriction is the class of the prescribed connected cover of \(E\), and is nonzero. This proves \eqref{eq:cover-restriction}.
\end{proof}

\begin{remark}[Order of choices]\label{rem:choices}
The sequence is
\[
 B\ \longrightarrow\ T\ \longrightarrow\ Y=T\times B
 \ \longrightarrow\ E\ \longrightarrow\ S.
\]
The elliptic curve is selected relative to \(Y\), not relative to \(S\times Y\). Constructing \(S\) afterward does not alter the hypotheses of Theorem~\ref{thm:AZ}. Also, \(f^*\lambda_S\) is a nontrivial torsion line bundle on \(E\), whereas \(f^*t_S=0\) in \(\NS(E)\). Confusing those two statements would invalidate the Bockstein argument.
\end{remark}

\section{The torsion-linking bracket and its full indeterminacy}\label{sec:detector}
Put \(X_0=S\times Y=S\times T\times B\), a sevenfold, and suppress pullback notation for its factors. Since \(2t_S=2z=0\), the bracket
\begin{equation}\label{eq:sevenfold-bracket}
 \br{t_S,2,z}_{\mor}\subset L^4H^7(X_0)
\end{equation}
is defined. Its middle class is the coefficient class \(2\in L^0H^0(X_0)\).

\subsection{A defining system with a singular-cohomology-zero value}
The following cochain calculation is performed first in a derived tensor-product model for the factors and then mapped to the multiplicative model of the product by the exterior-product map. No K\"unneth quasi-isomorphism is assumed. Choose closed representatives \(A\in\cC_S^2(1)\), \(B_0\in\cC_T^2(1)\), and \(G\in\cC_B^4(2)\) for \(t_S,t_T,\gamma\). Choose cochains \(U\in\cC_S^1(1)\), \(V\in\cC_T^1(1)\) satisfying
\begin{equation}\label{eq:torsion-nullhomotopies}
 dU=2A,\qquad dV=2B_0.
\end{equation}
Set
\begin{equation}\label{eq:defining-system}
 Z=B_0\boxtimes G,\qquad V_Y=V\boxtimes G,
 \qquad W=A\boxtimes V_Y-U\boxtimes Z.
\end{equation}
Then \(dV_Y=2Z\), so \(W\) represents a value \(w\) of \eqref{eq:sevenfold-bracket}.

\begin{lemma}\label{lem:w-properties}
The chosen value satisfies
\[
 2w=0,\qquad \cl(w)=0.
\]
\end{lemma}
\begin{proof}
The signs in \eqref{eq:defining-system} use \(|A|=2\), \(|U|=1\), and \(|Z|=6\). The derivation rule gives
\begin{equation}\label{eq:order-two}
 d(U\boxtimes V_Y)=2A\boxtimes V_Y-2U\boxtimes Z=2W.
\end{equation}
For the cycle class, put
\[
 W_0=A\boxtimes V-U\boxtimes B_0.
\]
It is closed, and \(W=W_0\boxtimes G\). Multiplicativity of realization and \(\cl(\gamma)=0\) imply
\begin{equation}\label{eq:factor-zero}
 \cl(w)=\cl([W_0])\boxtimes\cl(\gamma)=0.
\end{equation}
\end{proof}

\subsection{An additive restriction map annihilating all indeterminacy}
The full indeterminacy of \eqref{eq:sevenfold-bracket} is
\begin{equation}\label{eq:I0}
 I_0=t_S L^3H^5(X_0)+L^1H^1(X_0)z.
\end{equation}
Choose \(e_0\in E\), write \(s_0=f(e_0)\), and define
\[
 j=f\times\id_Y:E\times Y\longrightarrow X_0,\qquad
 s:Y\longrightarrow X_0,\quad y\longmapsto(s_0,y),
\]
with \(p:E\times Y\to Y\) the projection. The relevant map is the \emph{additive} homomorphism
\begin{equation}\label{eq:R}
 \mathcal R=j^*-p^*s^*:L^4H^7(X_0)\longrightarrow L^4H^7(E\times Y).
\end{equation}
It is not being asserted to be a ring homomorphism.

\begin{lemma}\label{lem:R-indeterminacy}
One has \(\mathcal R(I_0)=0\).
\end{lemma}
\begin{proof}
Both \(j^*t_S\) and \(s^*t_S\) vanish: the first by Lemma~\ref{lem:surface}, the second by restriction to a point. Thus \(\mathcal R\) kills \(t_SL^3H^5(X_0)\), without any hypothesis on the second factor in this product.

For the other summand, Lemma~\ref{lem:low-weight} and \(H^1(S;\Z)=0\) give
\[
 L^1H^1(S\times Y)=H^1(S\times Y;\Z)
 =\operatorname{pr}_Y^*H^1(Y;\Z).
\]
Hence every class in \(L^1H^1(X_0)z\) is pulled back from \(Y\). The two terms of \(\mathcal R\) agree on all such classes. This proves the assertion for the full subgroup \eqref{eq:I0}.
\end{proof}

\subsection{Evaluation on the chosen defining system}
\begin{lemma}\label{lem:R-value}
There is \(u\in L^1H^1(E)\) with \(\bar u\neq0\) modulo two such that
\begin{equation}\label{eq:Rw}
 \mathcal R(w)=-u\boxtimes z.
\end{equation}
\end{lemma}
\begin{proof}
Since \(f^*t_S=0\), choose a cochain \(D\in\cC_E^1(1)\) with \(dD=f^*A\), and put
\[
 K=f^*U-2D.
\]
Then \(dK=0\). Write \(u=[K]\in L^1H^1(E)\). Reduction of \(U\) modulo two represents the Bockstein preimage \(\alpha_S\), so
\begin{equation}\label{eq:u-reduction}
 \bar u=f^*\alpha_S\neq0
 \quad\text{in }L^1H^1(E;\F_2).
\end{equation}
Here \(L^1H^1(E)/2\to L^1H^1(E;\F_2)\) is an isomorphism because \(L^1H^2(E)=\NS(E)=\Z\) has no two-torsion. Thus \(u\) is nonzero modulo two in the quotient group required by Theorem~\ref{thm:AZ}.

Now restrict \eqref{eq:defining-system}. Using \(|D|=1\) and \(dV_Y=2Z\),
\begin{align*}
 j^*W
 &=dD\boxtimes V_Y-(K+2D)\boxtimes Z\\
 &=d(D\boxtimes V_Y)-K\boxtimes Z.
\end{align*}
Hence \(j^*w=-u\boxtimes z\). Restriction to \(\{e_0\}\times Y\) gives zero, since the restriction of \(u\) to a point is zero in degree one. Consequently \(s^*w=0\), proving \eqref{eq:Rw}.
\end{proof}

\begin{theorem}[Sevenfold nonvanishing]\label{thm:sevenfold}
The bracket \eqref{eq:sevenfold-bracket} does not contain zero. The value \(w\) has exact order two, has zero singular cohomology realization, and is nonzero modulo the full indeterminacy \(I_0\).
\end{theorem}
\begin{proof}
Equations~\eqref{eq:z-properties}, \eqref{eq:u-reduction}, and Theorem~\ref{thm:AZ} imply
\[
 \overline{u\boxtimes z}\neq0\quad\text{in }L^4H^7(E\times Y)/2.
\]
Thus \(\mathcal R(w)\neq0\), whereas Lemma~\ref{lem:R-indeterminacy} gives \(\mathcal R(I_0)=0\). Therefore \(w\notin I_0\). By Lemma~\ref{lem:indeterminacy}, the bracket is the coset \(w+I_0\) and does not contain zero. Lemma~\ref{lem:w-properties} now proves the assertions about order and cycle class.
\end{proof}

\begin{remark}[Why restriction does not trivialize the argument]\label{rem:restriction}
The restricted bracket on \(E\times Y\) has first class zero and therefore contains zero. We do \emph{not} deduce nonvanishing from the nonvanishing of that restricted bracket. Rather, \(\mathcal R\) annihilates the \emph{image of the global indeterminacy} while its value on the chosen global defining system is nonzero. The class \(u\) is a degree-one class on \(E\) that cannot arise by restriction from \(H^1(S;\Z)\), since the latter group is zero. This failure to extend is exactly what the calculation uses.
\end{remark}

\section{An eightfold with positive-degree classes}\label{sec:eightfold}
\subsection{The projective-line transfer}
Set
\[
 X=X_0\times\PP^1,\qquad \pi:X\longrightarrow X_0,
 \qquad h=c_1(\mathcal O_{\PP^1}(1)).
\]
Take the morphic classes
\begin{equation}\label{eq:eightfold-classes}
 a=\pi^*t_S\in L^1H^2(X),\quad
 b=2h\in L^1H^2(X),\quad
 c=\pi^*z\in L^3H^6(X).
\end{equation}
Their adjacent products vanish. Multiplying each null-homotopy in \eqref{eq:defining-system} by a closed representative of \(h\), using the external product with \(\PP^1\), gives the value
\begin{equation}\label{eq:wtilde}
 \widetilde w=\pi^*w\cdot h
 \in\br{a,b,c}_{\mor}\subset L^5H^9(X).
\end{equation}
In particular, this is a value obtained from a specified defining system, not only a class in the anticipated target.

The full indeterminacy is
\begin{equation}\label{eq:J}
 J=aL^4H^7(X)+L^2H^3(X)c.
\end{equation}
The pushforward \eqref{eq:projective-push} and the projection formula give
\begin{equation}\label{eq:push-J}
 \pi_*(J)\subseteq I_0,\qquad \pi_*(\widetilde w)=w.
\end{equation}
Explicitly, the two groups \(L^4H^7(X)\) and \(L^2H^3(X)\) push to \(L^3H^5(X_0)\) and \(L^1H^1(X_0)\), respectively. Thus \(\widetilde w\in J\) would imply \(w\in I_0\), contrary to Theorem~\ref{thm:sevenfold}. We have proved
\begin{equation}\label{eq:eightfold-nonzero}
 0\notin\br{a,b,c}_{\mor},\qquad
 2\widetilde w=0,\qquad \cl(\widetilde w)=0.
\end{equation}
The nonzero order-two assertion follows also from \(\pi_*\widetilde w=w\neq0\).

\subsection{Lawson degrees and invisibility to singular homology}
Since \(\dim_\C X=8\), define
\[
 x_1=D_X(a),\qquad x_2=D_X(b),\qquad x_3=D_X(c),
 \qquad \xi=D_X(\widetilde w).
\]
The complete degree conversion is
\begin{equation}\label{eq:degree-table}
 \begin{array}{c|c|c}
 \text{class}&\text{morphic group}&\text{Lawson group}\\ \hline
 a,b&L^1H^2(X)&L_7H_{14}(X)\\
 c&L^3H^6(X)&L_5H_{10}(X)\\
 \widetilde w&L^5H^9(X)&L_3H_7(X).
 \end{array}
\end{equation}
The Lawson indeterminacy is exactly
\begin{equation}\label{eq:lawson-final-indeterminacy}
 I_L=x_1\bullet L_4H_9(X)+L_6H_{13}(X)\bullet x_3.
\end{equation}
Duality and \eqref{eq:eightfold-nonzero} give
\[
 0\notin\br{x_1,x_2,x_3}_L,\qquad
 \xi\notin I_L,\qquad 2\xi=0,\qquad \Phi_{3,7}(\xi)=0.
\]
All groups displayed here are in the ordinary Lawson range.

The third class has \(\cl(c)=0\), so the singular cohomology bracket
\(\br{\cl(a),\cl(b),0}\) contains zero. To see this directly, represent its third class by the zero cochain, take the second product's null-homotopy to be zero, and choose any null-homotopy for the first product. Its defining value is then zero. By Theorem~\ref{thm:comparison}, the singular quotient class is the image of the nonzero morphic quotient class. Equivalently, \(\Phi(x_3)=0\) and
\[
 0\in\br{\Phi(x_1),\Phi(x_2),\Phi(x_3)}_{\sing}.
\]
Moreover, \(\Phi(\xi)=0\) holds before quotienting. This proves Theorem~\ref{thm:main}.

\subsection{A cycle-level representative}
The defining system has a concrete higher-Chow expression. Choose divisors \(D_S,D_T\) representing \(\lambda_S,K_T\), and rational functions \(f_S,f_T\) with
\[
 \operatorname{div}(f_S)=2D_S,\qquad
 \operatorname{div}(f_T)=2D_T.
\]
Let \(F_S\in z^1(S,1)\) and \(F_T\in z^1(T,1)\) be the standard admissible cubical chains for these rational functions, with boundary convention
\(\partial F_S=2D_S\), \(\partial F_T=2D_T\). Choose a cycle representative \(\Gamma\) of \(\gamma\) and a point \(P\in\PP^1\). The chain
\begin{equation}\label{eq:chow-chain}
 \Theta=(D_S\boxtimes F_T-F_S\boxtimes D_T)
       \boxtimes\Gamma\boxtimes[P]
\end{equation}
lies in \(z^5(X,1)\) and has boundary zero. All these products are external products on different factors; no unsupported internal intersection of unmoved chains is being used. The monoidal natural transformation \(\id\to Q^{\sst}\) of \cite[Theorem~4.11]{Heller} sends the factorwise motivic defining system determined by \(F_S,F_T,\Gamma,P\) to a defining system for the morphic bracket \(\br{a,b,c}_{\mor}\). Hence the morphic image of \([\Theta]\) is a value of that bracket and differs from the chosen value \(\widetilde w\) by the full indeterminacy \(J\). This weaker, naturality-based statement is all that is used later; no chain-level identification of two independently chosen models is required. The finite-coefficient higher-Chow--morphic comparison is stated explicitly in \cite[Theorem~2.8]{AZ}.

Formula~\eqref{eq:chow-chain} is therefore a geometric presentation of a source defining system. Its boundary cancellation alone does not prove nonvanishing; nonvanishing is supplied after comparison by the relative detector and the full-indeterminacy calculation.

\subsection{Infinitely many brackets on a fixed eightfold}
\begin{corollary}\label{cor:infinite}
The factors can be fixed so that \(X\) supports infinitely many nonzero integral triple Lawson brackets, invisible to singular homology, of the degrees in Theorem~\ref{thm:main}. Their specified values \(\xi_j\) can be chosen linearly independent modulo two in \(L_3H_7(X)\), with \(2\xi_j=0\) and \(\Phi(\xi_j)=0\). In particular, these values generate a subgroup isomorphic to \(\bigoplus_{j\geq1}\Z/2\) in the kernel of the Lawson cycle map.
\end{corollary}
\begin{proof}
By Lemma~\ref{lem:gamma}, choose homologically trivial \(\gamma_j\) whose reductions in \(A^2(B)/2\) are linearly independent. The fixed choice of \(T\) works for all of them, so the reductions of \(z_j=t_T\boxtimes\gamma_j\) are independent. The same elliptic curve \(E\) works for the whole group in \eqref{eq:AZ}, and the same surface \(S\), class \(u\), and linking cochain \(W_0\) can be used. Set \(w_j=[W_0]\boxtimes\gamma_j\), \(\widetilde w_j=h\pi^*w_j\), and \(\xi_j=D_X(\widetilde w_j)\).

The additive map \(\mathcal R\pi_*\) sends \(\widetilde w_j\) to \(-u\boxtimes z_j\), whose reductions modulo two are independent by Theorem~\ref{thm:AZ}. Thus the \(\xi_j\) are independent modulo two. Each has exact order two and zero image under the Lawson cycle map by the preceding construction, and its own bracket survives its full indeterminacy by the same argument. This argument does not identify the individual indeterminacies. Section~\ref{sec:nonzero-classes} will prove the stronger independence statement in a quotient by their sum.
\end{proof}

\section{Nonzero singular homology classes and fixed homological data}\label{sec:nonzero-classes}
The construction of Section~\ref{sec:eightfold} has a zero third singular homology class. This is not an essential limitation of the example. A square-zero projective factor permits a change of the third class without changing either the selected value or the full indeterminacy. Throughout this section all coefficients are integral.

\subsection{A perturbation of the third class}
Keep \(X,X_0,\pi,a,b,c,h,w,\widetilde w\) as in Section~\ref{sec:eightfold}. Fix an ample line bundle \(\Theta\) defining a principal polarization on \(B\), and put \(\vartheta:=c_1^{\mor}(\Theta)\in L^1H^2(B)\); we use the same symbol for its pullbacks. Define
\begin{equation}\label{eq:modified-third}
 c^+=c+a h\vartheta\in L^3H^6(X).
\end{equation}
The equalities \(2a=2c=0\) imply \(ab=bc^+=0\).

\begin{proposition}[Square-zero perturbation]\label{prop:perturbation}
The bracket \(\br{a,b,c^+}_{\mor}\) has exactly the same indeterminacy as \(\br{a,b,c}_{\mor}\), and \(\widetilde w\) is a value of both brackets. In particular,
\begin{equation}\label{eq:equal-cosets}
 \br{a,b,c^+}_{\mor}=\widetilde w+J
   =\br{a,b,c}_{\mor},
 \qquad 0\notin\br{a,b,c^+}_{\mor}.
\end{equation}
All three singular cohomology classes \(\cl(a),\cl(b),\cl(c^+)\) are nonzero.
\end{proposition}
\begin{proof}
Write \(G^{q,m}=L^qH^m(X)\). The new indeterminacy is
\[
 J^+=aG^{4,7}+G^{2,3}(c+a h\vartheta).
\]
For \(v\in G^{2,3}\), the additional term \(v a h\vartheta\) lies in \(aG^{4,7}\), because \(vh\vartheta\in G^{4,7}\). Hence \(J^+\subseteq J\). Substituting \(c=c^+-a h\vartheta\) gives the reverse inclusion. Thus \(J^+=J\).

The assertion about a selected value requires a defining system, not just this equality of indeterminacies. Use the representatives \(A,U,Z,V_Y\) from \eqref{eq:defining-system}, with \(dU=2A\) and \(dV_Y=2Z\). Let \(H\) and \(P_\vartheta\) be closed representatives of \(h\) and \(\vartheta\) in the respective factor models. In the exterior tensor-product model take
\begin{equation}\label{eq:perturbed-system}
 C^+=Z+AHP_\vartheta,\qquad
 U_{12}=UH,\qquad V_{23}=HV_Y+UH^2P_\vartheta.
\end{equation}
The letters \(A,U\) come from \(S\), \(H\) from \(\PP^1\), and \(P_\vartheta\) from \(B\). Products on different tensor factors use the usual graded signs. The displayed cochains satisfy
\[
 dU_{12}=A(2H),\qquad dV_{23}=(2H)C^+.
\]
The corresponding bracket representative is
\begin{equation}\label{eq:perturbed-value}
 AV_{23}-U_{12}C^+
   =H(AV_Y-UZ)+(AU-UA)H^2P_\vartheta.
\end{equation}
Here \(AU-UA\) is closed: both \(d(AU)\) and \(d(UA)\) equal \(2A^2\). Its product with \(H^2P_\vartheta\) represents zero because \(h^2=0\) in the morphic cohomology of \(\PP^1\). Thus the class of \eqref{eq:perturbed-value} is \(\widetilde w\). This argument does not assume either \(AU=UA\) or \(H^2=0\) at the cochain level. Lemma~\ref{lem:indeterminacy} and Theorem~\ref{thm:sevenfold} now give \eqref{eq:equal-cosets}.

Put \(\tau=\cl(t_S)\in H^2(S;\Z)\). It has exact order two, since the integral divisor cycle map identifies \(\NS(S)\) with a subgroup of \(H^2(S;\Z)\). The integral class \(\cl(\vartheta)\) is primitive. Indeed, a principal polarization gives a unimodular alternating form on \(H_1(B;\Z)\); in a symplectic basis its degree-two class is \(\sum_{i=1}^3 e_i\wedge f_i\), and therefore has a coefficient equal to one. The cohomology of \(B\) and \(\PP^1\) is torsion-free, so the integral K\"unneth theorem gives
\[
 \tau\boxtimes\cl(\vartheta)\boxtimes h\neq0
       \quad\text{in }H^6(S\times B\times\PP^1;\Z).
\]
Pullback to \(X\) remains injective, since projection off the \(T\)-factor has a section. Consequently
\begin{equation}\label{eq:nonzero-singular-triple}
 \cl(a)=\tau\neq0,\qquad
 \cl(b)=2h\neq0,\qquad
 \cl(c^+)=\tau h\cl(\vartheta)\neq0.
\end{equation}
The first and third classes have exact order two, and the middle class has infinite order.
\end{proof}

\begin{theorem}[Three nonzero singular homology classes]\label{thm:nonzero-classes}
On the eightfold \(X\), there are classes
\[
 x_1,x_2\in L_7H_{14}(X;\Z),\qquad x_3^+\in L_5H_{10}(X;\Z)
\]
with vanishing adjacent products, all three \(\Phi(x_i)\) nonzero, and
\[
 0\notin\br{x_1,x_2,x_3^+}_L\subset L_3H_7(X;\Z).
\]
The bracket contains a specified value \(\xi\) of exact order two satisfying \(\Phi(\xi)=0\). The corresponding singular homology Massey bracket, despite its three nonzero classes, contains zero.
\end{theorem}
\begin{proof}
Take \(x_1=D_X(a)\), \(x_2=D_X(b)\), and \(x_3^+=D_X(c^+)\). Proposition~\ref{prop:perturbation} gives the same value \(\xi=D_X(\widetilde w)\) and the same full Lawson indeterminacy as in \eqref{eq:lawson-final-indeterminacy}. Its exact order and vanishing image under the Lawson cycle map were proved in Section~\ref{sec:eightfold}. Equations~\eqref{eq:nonzero-singular-triple} and \eqref{eq:compatibility} show that \(\Phi(x_1),\Phi(x_2),\Phi(x_3^+)\) are all nonzero. Finally, the multiplicative realization sends the morphic defining system to a singular-cohomology defining system whose selected value is zero; Poincar\'e duality therefore gives \(\Phi(\xi)=0\) and shows that the singular-homology Massey bracket contains zero. This proves the assertion without using a zero singular-homology class.
\end{proof}

\subsection{Infinitely many secondary classes over one singular homology triple}
We can strengthen Corollary~\ref{cor:infinite} by imposing fixed, nonzero singular homology data and then removing \emph{all} the indeterminacies simultaneously. Choose the independent cycles \(\gamma_j\) from that corollary on the same \(B\), and put
\[
 c_j=\pi^*(t_T\boxtimes\gamma_j),\qquad
 c_j^+=c_j+a h\vartheta,\qquad
 \widetilde w_j=h\pi^*w_j.
\]
For brevity write
\begin{equation}\label{eq:sum-indeterminacy}
 J_j=aL^4H^7(X)+L^2H^3(X)c_j,
 \qquad J_\Sigma=\sum_{j\geq1}J_j\subset L^5H^9(X).
\end{equation}
Each sum means finite sums of elements. By Proposition~\ref{prop:perturbation}, \(J_j\) is also the full indeterminacy for \(\br{a,b,c_j^+}_{\mor}\). The individual groups \(J_j\) are not asserted to coincide.

\begin{theorem}[A common quotient over fixed nonzero singular homology data]\label{thm:fixed-data}
All the brackets \(\br{a,b,c_j^+}_{\mor}\) have the same three singular cohomology classes, namely the nonzero triple in \eqref{eq:nonzero-singular-triple}. Their specified values \(\widetilde w_j\) remain linearly independent modulo
\[
 J_\Sigma+2L^5H^9(X).
\]
In particular, their images generate a subgroup isomorphic to
\(\bigoplus_{j\geq1}\Z/2\) in \(L^5H^9(X)/J_\Sigma\).

Equivalently, put \(I_\Sigma=D_X(J_\Sigma)\), and let \(I_{\sing}\) be the ordinary indeterminacy of the one fixed singular homology triple. Then the comparison
\begin{equation}\label{eq:common-quotient-comparison}
 \frac{L_3H_7(X;\Z)}{I_\Sigma}
 \longrightarrow
 \frac{H_7(X;\Z)}{I_{\sing}}
\end{equation}
has a subgroup \(\bigoplus_{j\geq1}\Z/2\) in its kernel represented by these Lawson--Massey values. Each value also lies in the unquotiented kernel of the Lawson cycle map.
\end{theorem}
\begin{proof}
The equality of the singular cohomology triples follows from \(\cl(c_j)=0\). Define
\[
 \mathcal D=\mathcal R\pi_*:
 L^5H^9(X)\longrightarrow L^4H^7(E\times Y),
\]
where \(\mathcal R\) is \eqref{eq:R}. Projection along \(\PP^1\) sends \(J_j\) into
\[
 t_SL^3H^5(X_0)+L^1H^1(X_0)z_j,
 \qquad z_j=t_T\boxtimes\gamma_j.
\]
The proof of Lemma~\ref{lem:R-indeterminacy} applies to every \(z_j\); hence
\begin{equation}\label{eq:common-detector}
 \mathcal D(J_\Sigma)=0,\qquad
 \mathcal D(\widetilde w_j)=-u\boxtimes z_j.
\end{equation}
The reductions of \(z_j\) are independent by Theorem~\ref{thm:schreieder}, and Theorem~\ref{thm:AZ} preserves this independence after multiplication by the fixed nonzero \(\bar u\). If a finite combination with coefficients \(\epsilon_j\in\{0,1\}\) satisfies
\[
 \sum_j\epsilon_j\widetilde w_j\in J_\Sigma+2L^5H^9(X),
\]
apply \(\mathcal D\) and reduce modulo two. All \(\epsilon_j\) must be zero. Each \(\widetilde w_j\) has order two, so the asserted direct sum follows.

For every \(j\), the comparison theorem sends \(D_X(J_j)\) into the indeterminacy of the same singular triple. It therefore also sends their sum into that subgroup, making \eqref{eq:common-quotient-comparison} well-defined. Each \(\Phi(D_X\widetilde w_j)\) is zero before quotienting. This proves the kernel assertion.
\end{proof}

\begin{remark}
The point of \(J_\Sigma\) is that the independence assertion is stronger than a list of unrelated nonzero brackets. Even after imposing every indeterminacy relation arising from the whole family, the specified order-two values remain independent. On the other hand, this statement does not say that all integral singular secondary operations on \(X\) vanish. It concerns the one displayed singular triple and its fixed quotient class.
\end{remark}

\section{Transfer and varieties with ample canonical bundle}\label{sec:transfer}
The examples are not confined to products or to varieties with a projective-line factor. This section applies the ordinary projection formula to \emph{full Massey indeterminacy}. Pushforward itself is not claimed to preserve Massey brackets.

\subsection{An indeterminacy-compatible transfer principle}
Write \(\mathcal H^{q,m}(V)=L^qH^m(V;\Z)\). Let \(f:W\to V\) be a morphism of smooth projective varieties. Suppose that, in all the bidegrees involved, there is an additive, bidegree-preserving map
\[
 T:\mathcal H^{*,*}(W)\longrightarrow\mathcal H^{*,*}(V)
\]
satisfying the bimodule projection formula and a scalar trace identity:
\begin{equation}\label{eq:transfer-axioms}
 \begin{split}
 T(f^*r\cdot v)&=r\cdot T(v),\\
 T(v\cdot f^*r)&=T(v)\cdot r,\\
 Tf^*&=d\,\id
 \end{split}
 \qquad(d\in\Z).
\end{equation}
These conditions are conditions on cohomology groups, not a claim that \(T\) is multiplicative.

\begin{proposition}[Transfer of nonvanishing]\label{prop:transfer}
Let \(\br{a_1,a_2,a_3}_{\mor}=w+I\) be a defined triple bracket on \(V\), and let \(I'\) be the full indeterminacy of the pulled-back triple on \(W\). Then
\[
 T(I')\subseteq I.
\]
If \(d[w]\neq0\) in \(\mathcal H^{Q,M}(V)/I\), the bracket on \(W\) does not contain zero. In particular, a nonzero quotient class of exact order \(N\) survives whenever \(\gcd(d,N)=1\).

If additionally \(Nw=0\) and \(\cl(w)=0\), then \(f^*w\) has exact order \(N\) and zero singular cohomology realization whenever \([w]\) has exact order \(N\) and \(\gcd(d,N)=1\).
\end{proposition}
\begin{proof}
The two terms of \(I'\) are \(f^*a_1\mathcal H^{q_2+q_3,m_2+m_3-1}(W)\) and \(\mathcal H^{q_1+q_2,m_1+m_2-1}(W)f^*a_3\). Applying \eqref{eq:transfer-axioms} sends them into the respective terms of \(I\). Naturality supplies the value \(f^*w\) of the pulled-back bracket. If that bracket contained zero, its coset description would imply \(f^*w\in I'\), and hence \(dw\in I\), contrary to the assumption.

If a positive integer \(k\) kills \(f^*w\), then \(kdw=0\), and therefore \(kd[w]=0\). The order and coprimality assumptions force \(N\mid k\); since \(Nf^*w=0\), its exact order is \(N\). Vanishing of the cycle class follows by naturality.
\end{proof}

\begin{corollary}[Generically finite morphisms and blowups]\label{cor:odd-degree}
Let \(f:W\to V\) be a surjective generically finite morphism of degree \(d\) between smooth connected complex projective varieties of the same dimension. Then \(T=f_*\) satisfies \eqref{eq:transfer-axioms}. Consequently every example above survives pullback under an odd-degree such morphism. This includes every smooth projective birational morphism to an example, and in particular a blowup along a smooth center.
\end{corollary}
\begin{proof}
Morphic pushforward, defined using duality, satisfies the projection formula; see \cite[Proposition~2.4(3)]{HuLi}. The pushforward of the unit is \(d\), since the corresponding top-dimensional cycle pushes forward with generic degree \(d\). Thus \(f_*f^*r=d r\). Every nonzero quotient class used here has exact order two, so Proposition~\ref{prop:transfer} applies when \(d\) is odd.
\end{proof}

\begin{corollary}[Projective bundles without a chosen section]\label{cor:projective bundle}
Let \(p:\PP(\mathcal E)\to V\) be the projective bundle of a vector bundle of rank \(r+1\), and let \(\zeta=c_1(\mathcal O_{\PP(\mathcal E)}(1))\). Then
\begin{equation}\label{eq:bundle-transfer}
 T(\eta)=p_*(\zeta^r\eta)
\end{equation}
satisfies \eqref{eq:transfer-axioms} with \(d=1\). All the nonvanishing and singular-homology-invisibility statements therefore survive pullback to \(\PP(\mathcal E)\), whether or not the bundle admits a section.
\end{corollary}
\begin{proof}
The projective bundle formula \cite[Corollary~5.1]{HuLi} gives \(p_*(\zeta^r)=1\), and pushforward lowers weight by \(r\) and cohomological degree by \(2r\). Thus \eqref{eq:bundle-transfer} preserves bidegrees. The projection formula proves the other assertions in \eqref{eq:transfer-axioms}. Apply Proposition~\ref{prop:transfer}.
\end{proof}

\begin{corollary}[Products with arbitrary smooth projective factors]\label{cor:product-stability}
Let \(W\) be any nonempty smooth connected complex projective variety. Pullback along \(p:V\times W\to V\) preserves every nonzero bracket and every fixed-data family constructed above.
\end{corollary}
\begin{proof}
Choose a point \(w_0\in W(\C)\). The section \(i(v)=(v,w_0)\) satisfies \(i^*p^*=\id\). Naturality of the bracket therefore shows that a pulled-back bracket containing zero would force the original bracket to contain zero. The same retraction proves preservation of linear independence modulo the sums of the pulled-back indeterminacies. Alternatively, with \(r=\dim W\), the operator \(T(\eta)=p_*([w_0]\cdot\eta)\) satisfies \eqref{eq:transfer-axioms} with \(d=1\).
\end{proof}

\begin{proposition}[Transfer of the fixed-data family]\label{prop:transfer-family}
The fixed nonzero singular homology triple and the independent family of Theorem~\ref{thm:fixed-data} persist under any succession of projective bundle pullbacks and odd-degree generically finite morphisms as above. Independence holds modulo the sum of the new full indeterminacies and modulo twice the new target group.
\end{proposition}
\begin{proof}
The transfers compose; their scalar trace is the product of the odd degrees and is therefore odd. For each bracket, the new full indeterminacy maps into its original one. Hence the sum maps into \(J_\Sigma\). Composing this transfer with \(\mathcal D\) in \eqref{eq:common-detector} sends the pulled-back specified value to \(-d\,u\boxtimes z_j\). Reduction modulo two gives the same independent family as before. This proves the common-quotient independence claim.

The corresponding singular-(co)homology transfers satisfy the same trace identity. Neither a nonzero class of order two nor a class of infinite order is killed by multiplication by an odd integer. Thus none of the three fixed singular homology classes becomes zero after pullback. The singular-homology images of all selected values remain zero by naturality.
\end{proof}

\begin{remark}[Limits of transfer]
A general proper pushforward is not a map of cochain algebras, and the proof does not treat it as one. Also, the degree-one assertion concerns a morphism \emph{to} a known example. It does not assert birational invariance for arbitrary smooth projective models related only by a rational map. Such an assertion would require a separate descent argument.
\end{remark}

\subsection{Ample canonical examples in every dimension at least eight}
\begin{theorem}[General type applications]\label{thm:general-type}
For every integer \(d\geq8\), there is a smooth connected complex projective \(d\)-fold \(Z_d\) with ample canonical bundle and classes
\[
 y_1,y_2\in L_{d-1}H_{2d-2}(Z_d;\Z),\qquad
 y_{3,j}\in L_{d-3}H_{2d-6}(Z_d;\Z)\quad(j\geq1)
\]
whose adjacent products vanish and whose Lawson brackets do not contain zero. The three singular homology classes are nonzero and do not depend on \(j\). There are specified values
\begin{equation}\label{eq:general-type-values}
 \eta_j\in\br{y_1,y_2,y_{3,j}}_L
       \subset L_{d-5}H_{2d-9}(Z_d;\Z)
\end{equation}
with \(2\eta_j=0\) and \(\Phi(\eta_j)=0\), independent modulo two and modulo the sum of all the full Lawson indeterminacies. In particular they generate a subgroup \(\bigoplus_{j\geq1}\Z/2\) both in the unquotiented kernel of the Lawson cycle map and in the corresponding common quotient kernel.
\end{theorem}
\begin{proof}
Start with \(V_d=X\times\PP^{d-8}\). Corollary~\ref{cor:projective bundle} and Proposition~\ref{prop:transfer-family} transfer the entire fixed-data family to \(V_d\). Choose a sufficiently positive line bundle \(L\) on \(V_d\) so that \(L^{\otimes3}\) is very ample and \(K_{V_d}\otimes L^{\otimes2}\) is ample. A general section \(s\) of \(L^{\otimes3}\) has a nonempty smooth reduced zero divisor \(D\).

Take the cyclic triple cover
\begin{equation}\label{eq:cyclic-cover}
 f:Z_d=\Spec_{V_d}
       \bigl(\mathcal O_{V_d}\oplus L^{-1}\oplus L^{-2}\bigr)
       \longrightarrow V_d,
\end{equation}
with multiplication determined by \(s\). This is the standard cyclic-cover construction; see \cite[Section~3.5]{EV}. It is finite flat of degree three. Locally along \(D\) its equation is \(t^3=x_1\), where \(x_1\) is a smooth local equation for the divisor; hence \(Z_d\) is smooth there. Away from \(D\) the cover is \mbox{\'etale}. The section has valuation one along each component of \(D\), so its generic rational expression is not a cube. The cubic polynomial is irreducible, and the cover is connected. Finiteness over a projective variety makes \(Z_d\) projective.

Let \(R\) be its reduced ramification divisor. Then \(f^*D=3R\) and \(\mathcal O_{Z_d}(R)=f^*L\). The ramification formula gives
\begin{equation}\label{eq:canonical-cover}
 K_{Z_d}\cong f^*K_{V_d}\otimes\mathcal O_{Z_d}(2R)
             \cong f^*(K_{V_d}\otimes L^{\otimes2});
\end{equation}
compare \cite[Section~3.16]{EV}. A finite pullback of an ample line bundle is ample, proving the canonical-bundle assertion.

Degree three is odd. Proposition~\ref{prop:transfer-family} therefore preserves the full fixed-data family and its independence. The bidegrees of the three morphic classes are still \((q,m)=(1,2),(1,2),(3,6)\), and the value bidegree is \((5,9)\). Duality in dimension \(d\) sends them to precisely the Lawson groups displayed in the statement. All these groups are in the ordinary Lawson range. For completeness, the individual Lawson indeterminacy is
\[
 y_1\bullet L_{d-4}H_{2d-7}(Z_d)
 +L_{d-2}H_{2d-3}(Z_d)\bullet y_{3,j}.
\]
This is the dual of the full morphic indeterminacy, so no additional quotient terms have been omitted.
\end{proof}

\section{Motivic Massey values and higher Chow torsion}\label{sec:higher-chow}
The cycle presentation in Section~\ref{sec:eightfold} has consequences before passing to semi-topological cycles. We use Bloch's higher Chow groups \cite{Bloch}. Write
\[
 H^m_M(V,\Z(q))=\CH^q(V,2q-m)
\]
in the range used here. The motivic cycle complexes, or equivalently the corresponding weight-graded motivic mapping spectra, have a multiplicative comparison to the morphic models of Section~\ref{sec:models}; see \cite[Theorem~4.11 and Section~5.2]{Heller}. We use this comparison, not an identification of motivic and morphic groups with integral coefficients.

For clarity, the associative integral model on the motivic side is justified in the same way as Lemma~\ref{lem:linear-model}. The weight-zero coefficient mapping spectrum for \(\MZ\) has homotopy groups \(H_M^m(\Spec\C,\Z(0))\), which are \(\Z\) for \(m=0\) and zero otherwise. It is therefore equivalent to \(H\Z\), and its weight-graded mapping spectra are \(H\Z\)-linear. Applying the associative comparison of \cite{Shipley} supplies, for each fixed \(V\), a model \(\mathcal M_V\) with \(H^m(\mathcal M_V(q))=H_M^m(V,\Z(q))\). Naturality of exterior products is kept at the multiplicative-spectrum level, exactly as in Section~
ef{sec:models}. The motivic bracket means the Toda bracket of this fixed multiplicative theory, or its cochain Massey representative after choosing such a model. No arbitrary internal intersection of unmoved higher Chow chains is used.

\subsection{Exact order two in higher Chow groups}
Use the divisors \(D_S,D_T\) and admissible cubical degree-one chains \(F_S,F_T\) from \eqref{eq:chow-chain}. Thus
\[
 \partial F_S=2D_S,\qquad \partial F_T=2D_T.
\]
For the independent homologically trivial cycles \(\gamma_j\), choose cycle representatives \(\Gamma_j\). Let \(P\) be a point of \(\PP^1\), and set
\begin{equation}\label{eq:higher-chow-family}
 \Theta_j=(D_S\boxtimes F_T-F_S\boxtimes D_T)
          \boxtimes\Gamma_j\boxtimes[P]\in z^5(X,1).
\end{equation}

\begin{theorem}[Scalar-indecomposable higher Chow torsion]\label{thm:higher-chow}
The cycles \(\Theta_j\) define classes of exact order two in \(\CH^5(X,1)\). Their reductions are linearly independent modulo two, and their images under the composite
\[
 \CH^5(X,1)\longrightarrow L^5H^9(X)
        \xrightarrow{\cl}H^9(X;\Z)
\]
are zero. If \(\omega_j\in L^5H^9(X)\) denotes the morphic image of \([\Theta_j]\), then \(\omega_j\) is a value of \(\br{a,b,c_j}_{\mor}\), so \(\omega_j-\widetilde w_j\in J_j\).

Define the scalar-decomposable subgroup by
\begin{equation}\label{eq:scalar-decomposables}
 D^5(X)=\im\bigl(\CH^4(X)\otimes_{\Z}\C^\times
                   \longrightarrow\CH^5(X,1)\bigr).
\end{equation}
The classes \([\Theta_j]\) remain independent modulo \(D^5(X)+2\CH^5(X,1)\). In particular, their images generate \(\bigoplus_{j\geq1}\Z/2\) in the scalar-indecomposable quotient \(\CH^5(X,1)/D^5(X)\).
\end{theorem}
\begin{proof}
Closedness follows from the two boundary identities. There is also the explicit degree-two boundary calculation
\begin{equation}\label{eq:chow-order-two}
 \partial\bigl(F_S\boxtimes F_T\boxtimes\Gamma_j\boxtimes[P]\bigr)
       =2\Theta_j.
\end{equation}
The minus sign comes from the cubical degree one of \(F_S\). Products are on different variety factors, and products of admissible cubical chains are admissible. Thus \(2[\Theta_j]=0\).

By multiplicative naturality, the motivic-to-morphic comparison sends \([\Theta_j]\) to a value \(\omega_j\) of the morphic bracket \(\br{a,b,c_j}_{\mor}=\widetilde w_j+J_j\). The common detector \(\mathcal D=\mathcal R\pi_*\) kills every \(J_j\), hence
\[
 \mathcal D(\omega_j)=\mathcal D(\widetilde w_j)=-u\boxtimes z_j.
\]
The reductions of the right-hand side are linearly independent by Theorem~\ref{thm:AZ}; therefore the \(\omega_j\), and hence the \([\Theta_j]\), are independent modulo two. Since \(2[\Theta_j]=0\), each has exact order two. The singular realization of the displayed factorwise defining cycle contains the homologically trivial factor \(\Gamma_j\), so its singular cohomology image is zero.

For a connected smooth projective complex variety, \(\CH^1(X,1)=\mathcal O(X)^\times=\C^\times\); these classes are pulled back from \(\Spec\C\). Their morphic images vanish, since \(L^1H^1(\Spec\C)=0\). Multiplicativity therefore kills the entire subgroup \(D^5(X)\). If a finite \(\F_2\)-linear combination of the \([\Theta_j]\) belonged to \(D^5(X)+2\CH^5(X,1)\), its morphic image would be zero modulo two, contrary to the detector independence of the morphic images \(\omega_j\). This proves the last assertion.
\end{proof}

\begin{remark}[Meaning of indecomposable and relation to earlier results]\label{rem:chow-priority}
Here ``indecomposable'' refers exactly to the scalar-product subgroup \eqref{eq:scalar-decomposables}, not to every possible exterior product with cycles on other varieties. The cycles \(\Theta_j\) visibly admit exterior-product expressions. Infinite torsion in higher Chow groups modulo integers, and its separation from scalar-decomposable classes, are already studied in \cite{AZ}; our assertion is not a first existence theorem for such group-theoretic phenomena, nor a dimension-optimal one. The additional feature is that these particular cycles are specified Massey values surviving all the morphic indeterminacies, including the common quotient of Theorem~\ref{thm:fixed-data}. No vanishing of a Deligne regulator or an Abel--Jacobi invariant is asserted here.
\end{remark}

\subsection{Motivic brackets with the same nonzero singular cohomology data}
Let \(\widehat a\in H^2_M(X,\Z(1))\) be the class of \(D_S\), let \(\widehat b=2h\), and let
\[
 \widehat c_j=[D_T]\boxtimes\gamma_j\in H^6_M(X,\Z(3)),
 \qquad
 \widehat c_j^+=\widehat c_j+\widehat a h\widehat\vartheta.
\]
Here \(\widehat\vartheta\) denotes the polarization divisor class in motivic cohomology, with all pullbacks implicit. We have \(2\widehat a=2\widehat c_j=0\) already in the relevant Chow groups, because the two line bundles are of order two.

\begin{corollary}[Nonzero motivic secondary operations]\label{cor:motivic-brackets}
The integral motivic bracket
\begin{equation}\label{eq:motivic-bracket}
 \br{\widehat a,\widehat b,\widehat c_j^+}_M
       \subset H^9_M(X,\Z(5))=\CH^5(X,1)
\end{equation}
is defined and does not contain zero. It has \([\Theta_j]\) as a specified value. Its three singular cohomology classes are the same nonzero classes as in Proposition~\ref{prop:perturbation}, whereas that specified value has zero singular cohomology image.
\end{corollary}
\begin{proof}
Use the null-homotopies given by \(F_S\) and \(F_T\boxtimes\Gamma_j\). With the extra \(\PP^1\)-factor, they give the value \([\Theta_j]\) for the unmodified third class. The calculation \eqref{eq:perturbed-system}--\eqref{eq:perturbed-value} applies in an associative motivic model as well: the error term is a closed commutator multiplied by \(h^2\), and \(h^2=0\) in \(\CH^2(\PP^1)\). Thus the modified bracket also contains the specified value \([\Theta_j]\).

If zero lay in the motivic bracket, multiplicative naturality would send it to zero in \(\br{a,b,c_j^+}_{\mor}\), contradicting Proposition~\ref{prop:perturbation}. This checks nonvanishing modulo the \emph{full} motivic indeterminacy: its image lies in the full morphic indeterminacy, so none of its elements can cancel the given value. The assertions about singular cohomology classes follow from the same comparison.
\end{proof}

\section{Integral and two-local nonformality}\label{sec:nonformality}
A nonzero Massey quotient is an obstruction to replacing a multiplicative complex by its cohomology algebra with zero differential. This is a standard principle of secondary operations. We record carefully the coefficient ring to which it applies in the present examples.

\begin{definition}
A weight-graded associative differential graded algebra \(A\) over a commutative ring \(R\) is \emph{formal over \(R\)} if there is a zigzag of weight-preserving quasi-isomorphisms of associative differential graded \(R\)-algebras between \(A\) and \(H^*(A)\) equipped with zero differential.
\end{definition}

\begin{lemma}[The Massey obstruction]\label{lem:formality-obstruction}
A formal differential graded algebra has zero in every defined triple Massey bracket. If an integral triple bracket has a nonzero quotient class of exact order two, it remains nonzero after flat extension to \(\Z_{(2)}\).
\end{lemma}
\begin{proof}
A multiplicative quasi-isomorphism maps a nonempty triple bracket into the corresponding triple bracket and maps its indeterminacy onto the full target indeterminacy. Both brackets are cosets, so the inclusion is equality under the cohomology isomorphism. Thus a zigzag of quasi-isomorphisms preserves whether zero belongs to a bracket. In a zero-differential algebra, the adjacent products of a defined triple are literally zero, and choosing both null-homotopies to be zero produces the value zero.

Flat localization commutes with cohomology and with the images of multiplication that define the two indeterminacy summands. Consequently, the localized quotient is the localization of the integral quotient, and the chosen defining system maps to a defining system there. A nonzero element of order two cannot be killed by an odd denominator. Therefore its image in the \(\Z_{(2)}\)-localized quotient remains nonzero.
\end{proof}

\begin{theorem}[Nonformality of the cycle-cohomology algebras]\label{thm:nonformality}
The weight-graded associative morphic cochain algebra \(\cC_X\) is not formal over \(\Z\), and its derived extension to \(\Z_{(2)}\) is not formal over \(\Z_{(2)}\). The same statements hold for the integral motivic cochain algebra of \(X\). The morphic nonformality statements also hold for every variety obtained in Section~\ref{sec:transfer}, including the \(Z_d\) with ample canonical bundle.
\end{theorem}
\begin{proof}
For \(\cC_X\), the quotient class of \(\widetilde w\) is nonzero and has exact order two by Proposition~\ref{prop:perturbation} together with \eqref{eq:eightfold-nonzero}. Apply Lemma~\ref{lem:formality-obstruction}. For motivic cohomology, Corollary~\ref{cor:motivic-brackets} gives a nonzero Massey quotient represented by \([\Theta_j]\), and \eqref{eq:chow-order-two} shows that this quotient class has exact order two. The same lemma applies. Every transferred morphic example has a selected value and quotient class of exact order two by Proposition~\ref{prop:transfer}, so it has the same formality obstruction.
\end{proof}

\begin{remark}[What this does not imply]\label{rem:no-rational-formality}
These statements concern integral and two-local multiplicative cycle models. They do not imply rational nonformality: rationalization kills the torsion classes used here. Nor does reduction modulo two automatically produce a nonzero triple bracket; the middle class \(2h\) becomes zero, so the reduced triple contains zero. Finally, we do not assert formality of integral singular cochains. The singular cohomology assertion is that one specified ordinary Massey quotient vanishes, not that all integral topological secondary operations vanish.
\end{remark}

\section{Scope, dependencies, and remaining questions}\label{sec:scope}
Theorem~\ref{thm:main} establishes a nonzero integral Lawson bracket invisible to singular homology, with a selected value in the actual kernel of the Lawson cycle map. Theorem~\ref{thm:nonzero-classes} removes the zero-singular-homology-class restriction of that initial construction, and Theorem~\ref{thm:fixed-data} supplies infinitely many secondary classes over one fixed nonzero singular homology triple. In particular, the stronger requirement that all three singular-homology classes be nonzero is achieved by the perturbation \eqref{eq:modified-third}; thus the initial zero-class feature is not intrinsic to the construction.

The proof of the original torsion example still depends on the representation and duality theorems, Totaro's nondivisibility result, Schreieder's Enriques exterior-product theorem, and the elliptic exterior-product theorem of Alexandrou--Zhou. The last ingredient is cited in its available preprint form \cite{AZ}. None of the applications replaces that ingredient by an assumed Lawson K\"unneth formula. The added transfer argument uses established projection and projective bundle formulas \cite{FG,HuLi}; product stability is a direct retraction argument; the general-type realization uses the classical cyclic-cover construction \cite{EV}.

The contribution of the secondary calculation is the control of the full indeterminacy, its invariance under the third-class perturbation, and the common-quotient detector. The transfer and nonformality principles are familiar structural arguments applied here to those values. The higher Chow infinitude result is not asserted as a first or dimension-optimal existence result; its additional feature is its explicit realization by the same nonzero motivic and Lawson Massey values, as explained in Remark~\ref{rem:chow-priority}. No exhaustive priority claim for these applications is made.

Several qualifications remain essential. The exhibited values and the first and third singular homology classes have order two. All three singular homology classes are nonzero, but they are not all nontorsion. The chosen ordinary singular homology Massey bracket has zero quotient class; no claim is made that every integral singular Massey product on these varieties vanishes. The factors are selected by very-general existence theorems and an equivariant complete-intersection construction, not by numerical equations for a single explicit abelian threefold and cycle. No minimality claim for the dimension is made.

\begin{question}[The rational problem]\label{q:rational}
Does a smooth complex projective variety admit a defined triple Lawson--Massey product with rational coefficients whose quotient class is nonzero?
\end{question}
By Corollary~\ref{cor:rational-formality}, any such bracket would map to zero in the rational singular-homology Massey quotient. Rationalization kills the torsion classes of the present construction, so the new applications do not answer this question.

\begin{question}[Nontorsion singular homology classes]\label{q:nontorsion}
Can an integral nonzero Lawson bracket invisible to singular homology be constructed with all three singular homology classes of infinite order?
\end{question}
Theorem~\ref{thm:nonzero-classes} only requires the three integral singular homology classes to be nonzero. The stronger nontorsion requirement is not established here. Neither projective bundle transfer nor odd-degree pullback changes the torsion nature of the two relevant classes.

\appendix
\section{A compact verification of the torsion-linking system}\label{app:checks}
The following records the calculations on which the new example depends, separately from the existence of the cycle classes used in the construction.

On \(X_0=S\times Y\), the bidegrees of the three classes \((q,m)\) are \((1,2),(0,0),(3,6)\). Therefore the output bidegree is \((4,7)\), and the two null-homotopy bidegrees are \((1,1),(3,5)\). With \(dU=2A\), \(dV_Y=2Z\), one has
\[
 \begin{aligned}
 d(A\boxtimes V_Y-U\boxtimes Z)&=2A\boxtimes Z-2A\boxtimes Z=0,\\
 d(U\boxtimes V_Y)&=2(A\boxtimes V_Y-U\boxtimes Z).
 \end{aligned}
\]
The same sign convention yields
\[
 f^*U=K+2D,\quad dD=f^*A
 \quad\Longrightarrow\quad
 j^*W=d(D\boxtimes V_Y)-K\boxtimes Z.
\]
The detector kills the two full subgroups
\[
 t_SL^3H^5(X_0),\qquad L^1H^1(X_0)z,
\]
for different reasons: the torsion divisor restricts to zero, while all degree-one classes on \(X_0\) are pulled back from \(Y\). No decomposition of \(L^3H^5(X_0)\) is used.

After adding \(\PP^1\), the bidegrees of the three classes are \((1,2),(1,2),(3,6)\), the output is \((5,9)\), and the indeterminacy factors have bidegrees \((4,7),(2,3)\). Integration along \(\PP^1\) subtracts \((1,2)\) and sends these factors to \((3,5),(1,1)\), exactly the sevenfold indeterminacy factors. Duality in dimension eight sends the output to \((p,k)=(3,7)\), and the bidegrees of the three classes to \((7,14),(7,14),(5,10)\).

The logical chain is thus
\[
 \overline{\mathcal R(w)}\neq0,\quad \mathcal R(I_0)=0
 \ \Longrightarrow\ w\notin I_0
 \ \Longrightarrow\ \widetilde w\notin J
 \ \Longrightarrow\ \xi\notin I_L.
\]
The independent factorization \(W=W_0\boxtimes G\), with \(\cl(\gamma)=0\), proves vanishing of the selected singular cohomology value. This is stronger than, and distinct from, the elementary fact that a singular bracket with a zero third class contains zero.

\section{Checks for the perturbation and transfer}\label{app:application-checks}
For the perturbation, \(A,U,H,P_\vartheta\) have bidegrees \((1,2),(1,1),(1,2),(1,2)\). Thus \(AHP_\vartheta\) has bidegree \((3,6)\), while \((AU-UA)H^2P_\vartheta\) has bidegree \((5,9)\), as required. Only its cohomology class is set to zero: the equality used is \(h^2=0\), not strict commutativity of the integral cochains. For \(v\in L^2H^3\), the product \(vh\vartheta\) lies in \(L^4H^7\). This checks both the forward and reverse inclusions in \(J^+=J\).

For the fixed-data family, the same additive map \(\mathcal R\pi_*\) kills \(J_j\) for every \(j\), and hence kills the subgroup generated by them all. Reducing its values modulo two proves independence even after adding twice the entire target group. Thus the argument does not mistake different Massey quotients for one pre-existing common quotient.

For a transfer \(T\), the implication is
\[
 f^*w\in I'\ \Longrightarrow\ dw\in I.
\]
It uses the projection formula on the two complete indeterminacy summands and never applies a multiplicative-naturality statement to \(T\). In the projective bundle case, multiplication by \(\zeta^r\) raises bidegree by \((r,2r)\), and pushforward subtracts the same amount. For the cyclic triple cover, \(d=3\) is prime to the order of the quotient class. The dimension-dependent duality formula is always \((q,m)\mapsto(d-q,2d-m)\).

Finally, the higher Chow torsion assertion has its own boundary, \eqref{eq:chow-order-two}. It is not inferred merely from a torsion morphic image. Nonvanishing and independence are then detected by that image, while the scalar-decomposable subgroup maps to zero. The two-local formality obstruction follows from flat localization of the actual full-indeterminacy quotient; it is not an assertion about reducing the bracket modulo two.


\begin{thebibliography}{99}
\bibitem{Adams}
J.~F.~Adams,
\emph{Stable Homotopy and Generalised Homology},
Chicago Lectures in Mathematics, University of Chicago Press, Chicago, 1974.

\bibitem{Alexandrou}
T.~Alexandrou,
\emph{Torsion in Griffiths groups},
to appear in \emph{Algebraic Geometry}; \arxiv{2303.04083}, version~2, 2025; Theorem~6.1.

\bibitem{AZ}
T.~Alexandrou and L.~Zhou,
\emph{Torsion higher Chow cycles modulo \(\ell\)},
\arxiv{2503.20004}, 2025; Theorem~2.8 and Corollary~6.3.

\bibitem{BauesMuro}
H.-J.~Baues and F.~Muro,
\emph{Toda brackets and cup-one squares for ring spectra},
Comm. Algebra \textbf{37} (2009), no.~1, 56--82.
\doi{10.1080/00927870802241188}.

\bibitem{Bloch}
S.~Bloch,
\emph{Algebraic cycles and higher K-theory},
Adv. Math. \textbf{61} (1986), 267--304.

\bibitem{Catanese}
F.~Catanese,
\emph{Topological methods in moduli theory},
Bull. Math. Sci. \textbf{5} (2015), 287--449.
\doi{10.1007/s13373-015-0070-1}.

\bibitem{DGMS}
P.~Deligne, P.~Griffiths, J.~Morgan, and D.~Sullivan,
\emph{Real homotopy theory of K\"ahler manifolds},
Invent. Math. \textbf{29} (1975), 245--274.

\bibitem{Ekedahl}
T.~Ekedahl,
\emph{Two examples of smooth projective varieties with non-zero Massey products},
in \emph{Algebra, Algebraic Topology and their Interactions} (Stockholm, 1983),
Lecture Notes in Math., vol.~1183, Springer, 1986, pp.~128--132.
\doi{10.1007/BFb0075454}.

\bibitem{EV}
H.~Esnault and E.~Viehweg,
\emph{Lectures on Vanishing Theorems},
DMV Seminar, vol.~20, Birkh\"auser, Basel, 1992.
\url{https://page.mi.fu-berlin.de/esnault/books/esvibuch.pdf}.

\bibitem{Friedlander}
E.~M.~Friedlander,
\emph{Bloch--Ogus properties for topological cycle theory},
Ann. Sci. \'Ecole Norm. Sup. (4) \textbf{33} (2000), 57--79.

\bibitem{FriedlanderFiltrations}
E.~M.~Friedlander,
\emph{Filtrations on algebraic cycles and homology},
Ann. Sci. \'Ecole Norm. Sup. (4) \textbf{28} (1995), no.~3, 317--343.
\doi{10.24033/asens.1716}.

\bibitem{FG}
E.~M.~Friedlander and O.~Gabber,
\emph{Cycle spaces and intersection theory},
in \emph{Topological Methods in Modern Mathematics},
Publish or Perish, 1993, pp.~325--370.

\bibitem{FLcocycles}
E.~M.~Friedlander and H.~B.~Lawson, Jr.,
\emph{A theory of algebraic cocycles},
Ann. of Math. (2) \textbf{136} (1992), 361--428.

\bibitem{FLduality}
E.~M.~Friedlander and H.~B.~Lawson, Jr.,
\emph{Duality relating spaces of algebraic cocycles and cycles},
Topology \textbf{36} (1997), 533--565.

\bibitem{Heller}
J.~Heller,
\emph{Motivic strict ring spectra representing semi-topological cohomology theories},
Homology Homotopy Appl. \textbf{17} (2015), no.~2, 107--135.
\arxiv{1304.6288}.

\bibitem{HSS}
M.~Hovey, B.~Shipley, and J.~Smith,
\emph{Symmetric spectra},
J. Amer. Math. Soc. \textbf{13} (2000), no.~1, 149--208.
\doi{10.1090/S0894-0347-99-00320-3}.

\bibitem{Hu}
W.~Hu,
W.~Hu, \emph{Generalized Abel--Jacobi map on Lawson homology}, Amer. J. Math. \textbf{131} (2009), no.~5, 1241--1260.



\bibitem{HuLi}
W.~Hu and L.~Li,
\emph{Lawson homology, morphic cohomology and Chow motives},
Math. Nachr. \textbf{284} (2011), nos.~8--9, 1024--1047.


\bibitem{Kraines}
D.~Kraines,
\emph{Massey higher products},
Trans. Amer. Math. Soc. \textbf{124} (1966), 431--449.
\doi{10.1090/S0002-9947-1966-0202136-1}.

\bibitem{Lawson}
H.~B.~Lawson, Jr.,
\emph{Algebraic cycles and homotopy theory},
Ann. of Math. (2) \textbf{129} (1989), 253--291.

\bibitem{Massey1958}
W.~S.~Massey,
\emph{Some higher order cohomology operations},
in \emph{Symposium Internacional de Topolog\'ia Algebraica},
Universidad Nacional Aut\'onoma de M\'exico and UNESCO, Mexico City, 1958, pp.~145--154.

\bibitem{MasseyLinking}
W.~S.~Massey,
\emph{Higher order linking numbers},
J. Knot Theory Ramifications \textbf{7} (1998), no.~3, 393--414.
\doi{10.1142/S0218216598000206}.

\bibitem{May}
J.~P.~May,
\emph{Matric Massey products},
J. Algebra \textbf{12} (1969), 533--568.

\bibitem{MaySimplicial}
J.~P.~May,
\emph{Simplicial Objects in Algebraic Topology},
Van Nostrand Mathematical Studies, vol.~11, D.~Van Nostrand, Princeton, 1967.

\bibitem{Schreieder}
S.~Schreieder,
\emph{Infinite torsion in Griffiths groups},
J. Eur. Math. Soc. \textbf{27} (2025), no.~6, 2571--2601.
\doi{10.4171/JEMS/1419}.

\bibitem{SchwedeShipley}
S.~Schwede and B.~Shipley,
\emph{Stable model categories are categories of modules},
Topology \textbf{42} (2003), no.~1, 103--153.
\doi{10.1016/S0040-9383(02)00006-X}.

\bibitem{Shipley}
B.~Shipley,
\emph{\(H\Z\)-algebra spectra are differential graded algebras},
Amer. J. Math. \textbf{129} (2007), no.~2, 351--379.
\arxiv{math/0209215}.

\bibitem{Sullivan}
D.~Sullivan,
\emph{Infinitesimal computations in topology},
Publ. Math. Inst. Hautes \'Etudes Sci. \textbf{47} (1977), 269--331.
\doi{10.1007/BF02684341}.

\bibitem{Toda}
H.~Toda,
\emph{Composition Methods in Homotopy Groups of Spheres},
Annals of Mathematics Studies, vol.~49,
Princeton University Press, 1962.

\bibitem{Totaro}
B.~Totaro,
\emph{Complex varieties with infinite Chow groups modulo 2},
Ann. of Math. (2) \textbf{183} (2016), 363--375.
\arxiv{1502.02071}.
\end{thebibliography}
\end{document}